\documentclass[11pt]{amsart}

\usepackage{todonotes}

\usepackage{amsfonts, amssymb, amscd}

\usepackage[symbol]{footmisc}

\usepackage{amsmath}
\usepackage{bm}
\usepackage{verbatim}
\usepackage{mathrsfs}
\usepackage{graphicx}
\usepackage{tikz-cd}
\usepackage{subcaption}
\usepackage{listings}
\usepackage{subfiles}
\usepackage[toc,page]{appendix}
\usepackage{mathtools}
\usepackage{comment}
\usepackage{enumerate}
\usepackage{enumitem}
\usepackage[all]{xy}

\usepackage{graphicx}
\graphicspath{{images/}}

\usepackage{appendix}
\usepackage{hyperref}
\hypersetup{
	colorlinks=true,
	citecolor=red,
	linkcolor=blue,
	filecolor=magenta,      
	urlcolor=red,
}
\newcommand{\Cd}{\mathcal{D}}
\newcommand{\Ci}{\mathcal{I}}
\newcommand{\Co}{\mathcal{O}}

\newcommand{\Cs}{\mathcal{S}}
\newcommand{\Cx}{\mathcal{X}}

\newcommand{\Cl}{\mathcal{L}}
\newcommand{\Cb}{\mathcal{B}}

\newcommand{\Cv}{\mathcal{V}}
\newcommand{\Ct}{\mathcal{T}}
\newcommand{\Cp}{\mathcal{P}}

\newcommand{\M}{\mathbf{M}}

\newcommand{\Supp}{\mathrm{Supp}}
\newcommand{\vol}{\mathrm{vol}}

\theoremstyle{plain} 
\newtheorem{thm}{Theorem}[section] 
\newtheorem{lemma}[thm]{Lemma}
\newtheorem{prop}[thm]{Proposition}
\newtheorem{cor}[thm]{Corollary}
\theoremstyle{definition} 
\newtheorem{defn}[thm]{Definition} 
\newtheorem{conj}[thm]{Conjecture}
\theoremstyle{remark} 

\newtheorem{exa}[thm]{Example}

\begin{document}

	\title{On structures and discrepancies of klt Calabi--Yau pairs}
	
	\author{Junpeng Jiao}
	\email{jiao$\_$jp@tsinghua.edu.cn}
	\address{Mathematics Center, Tsinghua University, Beijing, China}
	

	\begin{abstract}
		We study the structures of klt Calabi--Yau pairs. We show that the discrepancies of log centers of all klt Calabi--Yau varieties with fixed dimension are in a finite set. As a corollary, we show that the index of 4-dimensional non-canonical Calabi--Yau variety is bounded.
	\end{abstract}
	
	\maketitle
	\tableofcontents
	Throughout this paper, we work over the complex number field $\mathbb{C}$.
	\section{Introduction}
	
	We say a pair $(X,B)$ is a Calabi--Yau pair if $(X,B)$ is a projective lc pair and $K_X+B\sim_{\mathbb{Q}} 0$, when $B=0$, we call $X$ a Calabi--Yau variety. 
	
	The Main result in this paper is the following.
	\begin{thm}\label{Main theorem 2}
		Fix $d\in \mathbb{N}$, then there exists $l\in \mathbb{N}$ depending only on $d$ such that:
		
		Let $X$ be a $d$-dimensional klt Calabi--Yau variety, $(Y,B_Y)\rightarrow X$ be a projective crepant birational morphism of $X$, then $lB_{Y, \geq 0}$ is an integral divisor, where $B_{Y, \geq 0}$ is the positive part of $B_Y$.  
	\end{thm}

	For the definition of crepant birational morphisms, see Section 2.1.

	In this paper we also investigate the indices of Calabi--Yau pairs. The index of a Calabi--Yau pair $(X,B)$ is defined as the smallest integer $m\geq 1$ such that $m(K_X+B)\sim 0$. It is expected that, for a fixed dimension of $X$ and a fixed coefficient set of $B$, the indices of Calabi--Yau pairs $(X,B)$ are bounded.
		
	The following is the index conjecture for canonical Calabi--Yau variety.
	
	\begin{conj}\label{index conjecture}
		Let $X$ be a canonical Calabi--Yau variety, then there exists $l\in \mathbb{N}$ depending only on $\mathrm{dim}(X)$ such that $lK_X\sim 0$.
	\end{conj}
	\begin{cor}\label{Main corollary 1}
		Assume Conjecture \ref{index conjecture} in dimension $d-1$,
		then there exists $l\in \mathbb{N}$ depending only on $d$ satisfying the following:
		
		If $X$ is a $d$-dimensional non-canonical klt Calabi--Yau variety, then $lK_X\sim 0$.
	\end{cor}
	
	The index conjecture has been extensively studied in prior research, including works such as \cite{Kaw86}, \cite{Mor86}, \cite{PS09}, \cite{HMX14}, \cite{Jia21}, \cite{Xu19a}, \cite{Xu19b}, \cite{JL21}, and \cite{Mas24}. Specifically, it was proven in dimension 2 by Prokhorov and Shokurov \cite[Corollary 1.11]{PS09}, for terminal 3-folds by Kawamata \cite{Kaw86} and Morrison \cite{Mor86}, for klt 3-folds by Jiang \cite{Jia21}, for lc pairs of dimension 3 and non-klt (lc) pairs of dimension 4 by Jiang and Liu \cite{JL21}, for lc pairs of dimension 4 with non-zero boundary by Xu \cite{Xu19a}, \cite{Xu19b}, and for smooth Calabi--Yau varieties of dimension 4 by Masamura \cite{Mas24}.
	
	We have the following unconditional result for 4-folds."
	\begin{cor}\label{Main corollary 2}
		There exists $l\in\mathbb{N}$ such that if $X$ is a $4$-dimensional non-canonical Calabi--Yau variety, then $lK_X\sim 0$.
	\end{cor}

	We also prove a result on the existence of fibration structure of klt Calabi--Yau pairs.
	\begin{thm}\label{Main theorem 1}
		Fix $d\in \mathbb{N}$, $\epsilon\in \mathbb{Q}$, and a DCC set $\Ci\in (0,1)\cap \mathbb{Q}$, then there exist $m,v\in \mathbb{N}$ depending only on $d,\epsilon$ and $r\in \mathbb{N}$ depending only on $d,\epsilon,\Ci$ satisfying the following:
		
		Suppose $(Y,B_Y)$ is a $d$-dimensional $\epsilon$-lc Calabi--Yau pair with $B_Y\neq 0$, then there exist a flop $(X,B)\dashrightarrow (Y,B_Y)$, a contraction $f:X\rightarrow Z$, a divisor $A$ on $X$, and a finite cover $\pi:W\rightarrow Z$ such that
		\begin{itemize}
			\item $\pi$ is \'etale in codimension 1,
			\item $Z$ is a Calabi--Yau variety,
			\item $f$ has no very exceptional divisor,
			\item $f$ has reduced fibers over codimension 1 points of $Z$,
			\item $f$ factors as a sequence of Fano contractions with length $\leq m$,
			\item $A_g:=A|_{X_g}$ is very ample with $\mathrm{vol}(A_g)\leq v$, where $X_g$ is a general fiber of $f$, 
			\item $X_W$ is isomorphic in codimension 1 with $X_g\times W$, where $X_W$ is the normalization of the main component of $X\times _Z W$, and
			\item $(X_W,B_W)\rightarrow W$ is generically trivial, where $B_W$ is the $\mathbb{Q}$-divisor such that $K_{X_W}+B_W$ is equal to the pullback of $K_X+B$.
		\end{itemize}
		Furthermore, if $\mathrm{coeff}(B)\subset \Ci$, then we can choose $\pi$ such that
		\begin{itemize}
			\item $\mathrm{deg}(\pi)=r$.
		\end{itemize}
	\end{thm}
	
	\cite[Theorem 1.1]{MW23} establishes that for a projective klt pair $(X,B)$ with nef anticanonical divisor $-(K_X+B)$, there exists an \'etale in codimension 1 cover $X_W\rightarrow X$ whose MRC fibration $X\rightarrow W$ is locally trivial. 
	As an application of Theorem \ref{Main theorem 1}, we given a different proof of \cite[Theorem 1.1]{MW23} for the Calabi--Yau pair case.
	
	\begin{cor}\label{Main corollary, decomposition for CY}
		Suppose $(Y,B_Y)$ is a klt Calabi--Yau pair, then there exist a flop $(X,B)\dashrightarrow (Y,B_Y)$, a contraction $f:X\rightarrow Z$, and a finite cover $\pi:W\rightarrow Z$ such that
		\begin{itemize}
			\item $\pi$ is \'etale in codimension 1,
			\item $Z$ is a Calabi--Yau variety with canonical singularities,
			\item $f$ has no very exceptional divisor,
			\item $f$ has reduced fibers over codimension 1 points of $Z$, 
			\item a general fiber $X_g$ of $f$ is rationally connected, and
			\item $X_W$ is isomorphic in codimension 1 with $X_g\times W$, where $X_W$ is the normalization of the main component of $X\times _Z W$.
		\end{itemize}
	\end{cor}

	\noindent\textbf{Sketch of the proof}. The proof of Theorem \ref{Main theorem 1} employs a similar approach to that of \cite[Theorem 3.1]{BDCS20}, using Mori fiber spaces to construct a tower of Fano fibrations. However, a key distinction lies in ensuring that the fibration has no very exceptional divisor. We first show that such a fibration can be constructed such that every vertical divisor dominates a divisor on the base.
	
	After constructing the fibration $f:(X,B)\rightarrow Z$, we encounter two cases: either $Z$ reduces to a closed point, or $K_Z\sim_{\mathbb{Q}} 0$. In the former scenario, $X$ becomes rationally connected, allowing us to invoke \cite[Theorem 1.6]{Bir23}. If $Z$ is not a closed point, then $K_{X/Z}+B\sim_{\mathbb{Q}} 0$, indicating that the moduli $\mathbf{b}$-divisor of $f$ is $0$. Consequently, $X\rightarrow Z$ is generically isotrivial, implying that after a finite base change the generic fiber will be a trivial fibration. Since the boundary part of $f:(X,B)\rightarrow Z$ is also $0$, for any prime divisor $P$ on $Z$, every irreducible component of $f^*P$ is an lc place of $(X,B+f^*P)$ over the generic point of $P$. Also because $f$ is generically isotrivial, $(X,B+f^*P)$ is crepant birationally equivalent to $(F\times Z,B_F\times Z+F\times P)$ over an analytic neighborhood of the generic point of $P$, where $(F,B_F)$ is a general fiber of $f$ and is klt. Note there is only one lc place of $(F\times Z,B_F\times Z+F\times P)$ dominates $P$, then $f^*P$ has only one component and $X\rightarrow Z$ has no very exceptional divisor.
	
	Theorem \ref{Main theorem 2} stands as an application derived from Theorem \ref{Main theorem 1}. Assuming $(Y,B_Y)\rightarrow X$ is a terminalization of $X$, then by \cite{HMX14}, $(Y,B_Y)$ is $\epsilon$-lc. Let $f:Y\rightarrow Z$ be the generically isotrivial fibration defined in Theorem \ref{Main theorem 1} and $(F,B_F)$ a general fiber of $f$. We use the fact that $f$ is isotrivial and has no very exceptional divisor to prove that the numerical Iitaka dimension of $B_F$ is equal to the numerical Iitaka dimension of $B_Y$, which is 0. Consequently, $(Y,(1+\delta)B_Y)$ has a good minimal model, $Y\dashrightarrow Y^m$, over $Z$, with $B_Y$ contracted by $Y\dashrightarrow Y^m$, and $K_{Y^m}\sim_{\mathbb{Q}} 0$. Because a general fiber $F^m$ of $Y^m\rightarrow Z$ is rationally connected and $\epsilon$-lc, it is bounded in codimension 1 by \cite[Theorem 1.6]{Bir23}. We use the boundedness to show that the Cartier index of $F^m$ is bounded, then the Cartier index of $(F,B_F)$ is bounded and the coefficients of $B_Y$ are in a finite set.

	\section{Preliminaries}
	\subsection{Notations and basic definitions}  We will use the same notation as in \cite{KM98} and \cite{Laz04}.
	
	For a birational morphism $f: Y\rightarrow X$ and a $\mathbb{Q}$-divisor $B$ on $X$, $f_*^{-1}(B)$ denotes the strict transform of $B$ on $Y$, and $\mathrm{Exc}(f)$ denotes the sum of the reduced exceptional divisors of $f$. Given two $\mathbb{Q}$-divisors $A,B$, $A\sim_{\mathbb{Q}} B$ means that there is an integer $m>0$ such that $m(A-B)\sim 0$. For a $\mathbb{Q}$-divisor $D$, we write $D=D_{\geq 0}-D_{\leq 0}$ as the difference of its positive and negative parts. Let $D:=\sum a_i D_i$ and $D':=\sum a'_i D_i$ be two $\mathbb{Q}$-divisors, then $D\wedge D':= \sum \min \{a_i,a'_i\} D_i$.
	
	A sub-pair $(X,B)$ consists of a normal variety $X$ and a $\mathbb{Q}$-divisor $B$ on $X$ such that $K_X+B$ is $\mathbb{Q}$-Cartier. We call $(X,B)$ a pair if in addition $B$ is effective. If $g: Y\rightarrow X$ is a birational morphism and $E$ is a divisor on $Y$, the discrepancy $a(E,X,B)$ is $-\mathrm{coeff}_{E}(B_Y)$, where $K_Y+B_Y :=g^*(K_X+B) $. Given $\epsilon\in (0,1)$, a sub-pair $(X,B)$ is called sub-klt (sub $\epsilon$-lc, sub-lc, sub-terminal) if for every birational morphism $Y\rightarrow X$ as above, $a(E,X,B)>-1$ ($\geq -1+\epsilon$, $\geq -1$, $>0$) for every divisor $E$ on $Y$. A pair $(X,B)$ is called klt ($\epsilon$-lc, lc, terminal) if $(X,B)$ is sub-klt ($\epsilon$-lc, sub-lc, sub-terminal) and $(X,B)$ is a pair.
	
	Given a sub-pair $(X,B)$, we call a prime divisor $P$ over $X$ an lc place (log place, non-terminal place) of $(X,B)$ if its discrepancy $a(P,X,B)$ is $=-1$ ($\in [-1,0)$, $\in [-1, 0]$ and $P$ is exceptional over $X$). A closed subvariety of $X$ is called an lc center (log center, non-terminal center) of $(X,B)$ if it is the image of an lc place (log place, non-terminal place).
	
	Let $(X,B),(Y,B_Y)$ be two sub-pairs and $h:Y\rightarrow X$ a birational morphism, we say $(Y,B_Y)\rightarrow (X,B)$ is a crepant birational morphism if $K_Y+B_Y\sim_{\mathbb{Q}}h^*(K_X+B)$ and $h_*B_Y=B$. Two pairs $(X_i,B_i),i=1,2$ are crepant birationally equivalent if there is a sub-pair $(Y,B_Y)$ and two crepant birational morphisms $(Y,B_Y)\rightarrow (X_i,B_i),i=1,2$. A birational map $(X_1,B_i)\dashrightarrow (X_2,B_2)$ is called a flop if it induces crepant birationally equivalence and $X_1\dashrightarrow X_2$ is an isomorphism in codimension 1.
	Let $(X,B)$ be a klt pair, a projective crepant birational morphism $(Y,B_Y)\rightarrow (X,B)$ is called a terminalization of $(X,B)$ if $(Y,B_Y)$ is terminal.
	
	A generalized pair $(X,B+\M_X)$ consists of a normal variety $X$ equipped with a projective morphism $X\rightarrow U$, a birational morphism $f:X'\rightarrow X$, a $\mathbb{Q}$-boundary $B$, and a $\mathbb{Q}$-Cartier divisor $\M_{X'}$ on $X'$ such that $K_{X}+B+\M_X$ is $\mathbb{Q}$-Cartier, $\M_{X'}$ is nef over $U$, and $\M_X=f_*\M_{X'}$. Let $B'$ be the $\mathbb{Q}$-divisor such that $K_{X'}+B'+\M_{X'}=f^*(K_X+B+\M_X)$, we call $(X,B+\M_X)$ generalized klt ($\epsilon$-lc, lc), if $(X',B')$ is sub-klt (sub-$\epsilon$-lc, sub-lc). Let $P$ be a prime divisor $P$ over $X$, we define the generalized discrepancy by $a(P,X,B+\M_X):=a(P,X',B')$.
	When $U$ is a point we drop it by saying $X$ is projective.
	
	A contraction is a projective morphism $f:X\rightarrow Z$ with $f_*\mathcal{O}_{X}=\mathcal{O}_Z$, hence it is surjective with connected fibers. A contraction $f:X\rightarrow Z$ is called a Fano contraction if $-K_X$ is ample over $Z$. Suppose $f:X\rightarrow Z$ is birational morphism and a contraction, we say $f$ is divisorial if the exceptional locus of $f$ is a divisor, we say $f$ is small if its exceptional locus has codimension $\geq 2$. A fibration means a contraction $X\rightarrow Z$ such that $\mathrm{dim}(X)>\mathrm{dim}(Z)$.
	
	Let $X\rightarrow Z$ be a fibration and $R$ a $\mathbb{Q}$-divisor on $X$, we write $R=R_v+R_h$, where $R_v$ is the vertical part and $R_h$ is the horizontal part. Given a contraction $f:X\rightarrow Z$ between normal varieties, a prime divisor $P$ on $X$ is called very exceptional over $Z$ if $P$ is vertical over $Z$ and over the generic point of any prime divisor $Q$ on $Z$ we have $\Supp(f^*Q)\not \subset \Supp(P)$.
	
	Let $X$ be a variety, an open subset $U\subset X$ is called big if the codimension of $X\setminus U$ is $\geq 2$.
	
	For a scheme $X$, a stratification of $X$ is a disjoint union $\coprod_i X_i$ of finitely many locally closed subschemes $X_i\hookrightarrow X$ such that the morphism $\coprod_i X_i\rightarrow X$ is both a monomorphism and surjective.
	
	\begin{defn}\label{slc pair}
		A \textit{semi-pair} $(X,\Delta)$ consists of a reduced quasi-projective scheme of pure dimension and a $\mathbb{Q}$-divisor $\Delta\geq 0$ on $X$ satisfying the following conditions:
		\begin{itemize}
			\item $X$ is $S_2$ with nodal codimension one singularities,
			\item no component of $\Supp (\Delta)$ is contained in the singular locus of $X$, and
			\item $K_{X}+\Delta$ is $\mathbb{Q}$-Cartier.
		\end{itemize}
		We say that $(X,\Delta)$ is \textit{semi-log canonical (slc)} if in addition we have:
		\begin{itemize}
			\item if $\pi:X^\nu \rightarrow X$ is the normalization of $X$ and $\Delta^\nu$ is the sum of the birational transform of $\Delta$ and the conductor divisor of $\pi$, then every irreducible component of $(X^\nu,\Delta^\nu)$ is lc. 
			We call $(X^\nu,\Delta^\nu)$ the normalization of $(X,\Delta)$.
		\end{itemize}
	\end{defn}
	
	\subsection{Families of pairs}
	The definition of families of projective pairs comes from \cite[\S 4]{Kol23}, in this paper we mainly deal with the case when the base is smooth.
	\begin{defn}
		Let $S$ be a reduced scheme and $n$ a natural number. A family of projective pairs of dimension $n$ over $S$ is an object
		$$f:(X,B)\rightarrow S,$$
		consisting of a morphism of schemes $f:X\rightarrow S$ and an effective $\mathbb{Q}$-divisor $B$ satisfying the following properties:
		\begin{itemize}
			\item $f$ is projective, flat, of finite type, of pure relative dimension $n$, with geometrically reduced fibers,
			\item every irreducible component $D_i\subset \Supp(B)$ dominates an irreducible component of $S$ and all nonempty fibers of $\Supp(B)\rightarrow S$ have pure dimension $n-1$. In particular, $\Supp(B)$ does not contain any irreducible component of any fiber of $f$, and
			\item the morphism $f$ is smooth at generic points of $X_s\cap \mathrm{Supp}(D)$ for every $s\in S$.
		\end{itemize}
		
		We say a family of projective pairs $f:(X,B)\rightarrow S$ is \textit{well-defined} if further
		\begin{itemize}
			\item there exists an open subset $U\subset X$ such that
			\begin{itemize}
				\item codimension of $X_s\setminus U_s$ is $\geq 2$ for every $s\in S$, and
				\item $B|_U$ is $\mathbb{Q}$-Cartier.
			\end{itemize}
		\end{itemize}
		
		Let $f:(X,B)\rightarrow S$ be a well-defined family of projective pairs over a reduced scheme $S$, we say $f$ is \textit{locally stable} if it satisfies the following conditions:
		\begin{itemize}
			\item $K_{X/S}+B$ is $\mathbb{Q}$-Cartier, and
			\item $(X_s,B_s)$ is an slc pair for every $s\in S$.
		\end{itemize}
		We say $f$ is \textit{stable} if further
		\begin{itemize}
			\item $K_{X/S}+B$ is ample over $S$.
		\end{itemize}
	\end{defn}
	According to \cite[Theorem-Definition 4.3]{Bir23}, when $S$ is normal, a family of projective family of pairs is naturally well-defined. 
	\begin{lemma}[{\cite[Corollary 4.55]{Kol23}}]\label{Kol23, Corollary 4.55}
		Let $S$ be a smooth variety, $(X,B)$ a pair and $f:(X,B)\rightarrow S$ a morphism. Then $f:(X,B)\rightarrow S$ is locally stable if and only if $(X,B+f^*D)$ is lc for every snc divisor $D\subset S$.
	\end{lemma}
	\begin{defn}
		Fix $d,n\in\mathbb{N}$, $v\in \mathbb{Q}^{>0}$, and a vector $\alpha=(a_1,...,a_m)$ with positive rational coordinates. Given a reduced scheme $S$, a strongly embedded $(d,\alpha,v,\mathbb{P}^n)$-marked locally stable family
		$$f:(X\subset \mathbb{P}^n_S,B)\rightarrow S$$
		is a stable morphism $f:(X,B)\rightarrow S$ together with a closed embedding $g:X\hookrightarrow \mathbb{P}^n_S$ such that
		\begin{itemize}
			\item $B=\sum a_i D_i$, where $D_i$ are irreducible component of $\Supp(B)$,
			\item $f=\pi\circ g$ where $\pi$ denotes the projection $\mathbb{P}^n_S\rightarrow S$, 
			\item letting $\Cl=g^*\Co_{\mathbb{P}^n_S}(1)$, we have $R^qf_*\Cl \cong R^q\pi_*\Co_{\mathbb{P}^n_S}(1)$ for each $q$, and
			\item $\vol(K_{X_s}+B_s)=v$ for each $s\in S$.
		\end{itemize}
		
		Consider the moduli functor $\mathcal{E}^s\mathcal{MLSP}_{d,\alpha,v,\mathbb{P}^n}$ of strongly embedded $(d,\alpha,v,\mathbb{P}^n)$-marked locally stable family from the category of reduced schemes to the category of sets by setting
		$$\mathcal{E}^s\mathcal{MLSP}_{d,\alpha,v,\mathbb{P}^n}(S)=\{\text{strongly embedded }(d,\alpha,v,\mathbb{P}^n)\text{-locally stable families over }S\}.$$
	\end{defn}
	\begin{thm}\label{Kol23, 8.5}
		The functor $\mathcal{E}^s\mathcal{MLSP}_{d,\alpha,v,\mathbb{P}^n}$ is represented by a reduced separated scheme $\mathrm{E}^s\mathrm{MLSP}_{d,\alpha,v,\mathbb{P}^n}$.
		\begin{proof}
			This is Theorem 7.2 in the first arxiv version of \cite{Bir20}, see also \cite[8.5]{Kol23}.
		\end{proof}
	\end{thm}
	\subsection{Boundedness of pairs}
	\begin{defn}
		Fix $d\in\mathbb{N}$. Let $\mathscr{S}$ be a set of $d$-dimensional pairs, we say $\mathscr{S}$ is bounded if there exists $v>0$ such that for any $(X,B)\in \mathscr{S}$, there exists a very ample divisor $A$ on $X$ such that $A^d\leq v$. We say $\mathscr{S}$ is log bounded if there exist $v,v'>0$ such that for any $(X,B)\in \mathscr{S}$, there exists a very ample divisor $A$ on $X$ such that $A^d\leq v$ and $\Supp(B).A^{d-1}\leq v'$. 
	\end{defn}

	By boundedness of Chow variety, see \cite[\S 1.3]{Kol96}, a set of $d$-dimensional pairs $\mathscr{S}$ is bounded (log bounded) if and only if there exists a flat morphism $\Cx\rightarrow \Cs$ (a flat morphism $\Cx\rightarrow \Cs$ with a divisor $\Cb$ on $\Cx$ which is flat over $\Cs$) over a scheme of finite type, such that for every $(X,B)\in \mathscr{S}$, there exists a closed point $s\in \Cs$ such that $X\cong \Cx_s$ ($(X,\Supp(B))\cong (\Cx_s,\Cd_s)$).
	\begin{lemma}\label{a general fiber has a very ample divisor implies bounded moduli space}
		Fix $d,l\in \mathbb{N}$ and $v\in \mathbb{Q}^{>0}$. Then there is a family of projective pairs $(\Cx,\Cb)\rightarrow \Cs$ satisfying the following:
		
		Suppose $f:X\rightarrow Z$ is contraction between normal varieties with a general fiber $X_g$, $A$ is a divisor on $X$, and $B$ is a $\mathbb{Q}$-divisor on $X$ such that
		\begin{itemize}
			\item $\mathrm{dim}(X_g)=d$,
			\item $A_g$ is very ample, 
			\item $lB$ is integral, 
			\item $(X_g,B_g)$ is lc,
			\item $A_g^d\leq v$, and
			\item $B_g.A_g^{d-1}\leq v$.
		\end{itemize}
		Then there exist an open subset $U\hookrightarrow Z$ and a morphism $U\rightarrow \Cs$ such that
		$(X_U,B_U):=(X,B)\times_Z U\rightarrow U$ is isomorphic to the base change of $(\Cx,\Cb)\rightarrow \Cs$ by $U\rightarrow \Cs$.
		\begin{proof}
			Because $A_g$ is very ample and $\mathrm{vol}(A_g)=v$, $X_g$ is in a bounded family. Also because $lB$ is integral and $B_g.A_g^{-d}\leq v$, then $(X_g,B_g)$ is log bounded. By log boundedness, there exist $r\in \mathbb{N}$ and $u\in \mathbb{R}^{>0}$ depending only on $d,l,v$ such that $K_{X_g}+B_g+rA_g$ is ample, $\vol(K_{X_g}+B_g+rA_g)=u$, and $rA_g$ is very ample without higher cohomology. After replacing $Z$ by an open subset and $rA$ by a general member of $|rA|$, we may assume $(X,B+rA)$ is lc and $K_X+B+rA$ is ample over $Z$. We replace $v$ by $(1+r)v$ and $B$ by $B+rA$, then $K_X+B$ is ample over $Z$.
			
			Since $lB$ is integral, by log boundedness, there are only finitely many combinations of the coefficient sets of $B$. Then to prove the result, we may assume there exists a fixed vector $\alpha=(a_1,...,a_m)$ of rational numbers such that $B=\sum a_i D_i$, where $D_i\subset \Supp(B)$ are irreducible components.

			Let $U\subset Z$ be a smooth open subset such that $(X,B)\rightarrow Z$ has a fiberwise log resolution $g:Y\rightarrow X$ over $U$. Then $(Y,B_Y)\rightarrow Z$ is log smooth over $U$, where $B_Y$ is the $\mathbb{Q}$-divisor such that 
			$$K_Y+B_Y\sim_{\mathbb{Q}} g^*(K_X+B).$$
			Define $(Y_U,B_{Y_U}):=(Y,B_Y)\times_Z U,(X_U,B_U):=(X,B)\times_Z U$ and denote the natural morphism $Y\rightarrow Z$ by $f_Y$, by log smoothness, we have $(Y_U,B_{Y_U}+(f_Y|_{Y_U})^*D)$ is sub-lc for every snc divisor $D\subset U$, then $(X_U,B_U+(f|_U)^*D)$ is lc. By lemma \ref{Kol23, Corollary 4.55}, $(X_U,B_U)\rightarrow U$ is locally stable. Also because $K_X+B$ is ample over $U$, then $(X_U,B_U)\rightarrow U$ is a stable morphism.
			
			By log boundedness of $(X_g,B_g)$ and the fact that $(X_U,B_U)\rightarrow U$ is a stable morphism, there exists $m$ depending only on $d,l,v$ such that $m(K_{X_U}+B_U)$ is relatively very ample over $U$. After replacing $U$ by an open subsets, we may assume $(f|_U)_*\Co_{X_U}(m(K_{X_U}+B_U))$ is free and defines an embedding $X_U\hookrightarrow \mathbb{P}^n_U$, where $n\in \mathbb{N}$ depends only on $d,l,v$. Then $(X_U\hookrightarrow \mathbb{P}^n_U,B_U)\rightarrow U$ is a strongly embedded $(d,\alpha,u,\mathbb{P}^n)$-locally stable families over $S$.
			
			According to Theorem \ref{Kol23, 8.5}, the functor $\mathcal{E}^s\mathcal{MLSP}_{d,\alpha,u,\mathbb{P}^n}$ is represented by a reduced separated scheme $\mathrm{E}^s\mathrm{MLSP}_{d,\alpha,u,\mathbb{P}^n}$. We define $\Cs:=\mathrm{E}^s\mathrm{MLSP}_{d,\alpha,u,\mathbb{P}^n}$ and $(\Cx,\Cb)\rightarrow \Cs$ be the corresponding universal family, the result follows.
		\end{proof}
	\end{lemma}
	\begin{lemma}\label{first step of relative MMP}
		Fix $d,r\in \mathbb{N}$. Let $\Cp$ be a log bounded set of $d$-dimensional projective pairs $(X,B+\Delta)$ such that $X$ is $\mathbb{Q}$-factorial, $(X,B)$ is klt, and $rB$ is integral. Then there exist a projective locally stable morphism $(\Cx,\Cb)\rightarrow \Cs$ over a scheme of finite type, a divisor $\Cd$ on $\Cx$ which is flat over $\Cs$, and a dense subset $\Cs'\subset \Cs$ such that 
		\begin{itemize}
			\item every irreducible component of $\Cs$ is smooth,
			\item $\Cx$ is $\mathbb{Q}$-factorial and klt, 
			\item for every $(X,B+\Delta)\in \Cp$, there exists a closed point $s\in \Cs'$ such that $(X,B+\mathrm{red}(\Delta))\cong (\Cx_s,\Cb_s+\Cd_s)$, and
			\item for every $s\in \Cs'$, there exists $(X,B+\Delta)\in \Cp$ such that $(X,B+\mathrm{red}(\Delta))\cong (\Cx_s,\Cb_s+\Cd_s)$.
		\end{itemize}
		\begin{proof}
			Fix $(X,B+\Delta)\in \Cp$.
			By the definition of log boundedness, there exist a contraction $V\rightarrow T$ over a scheme of finite type and reduced divisors $C,D$ on $V$ which are flat over $T$ and depend only on $\Cp$ such that $X$ is the fiber of $V\rightarrow T$ over $t\in T$, $\Supp(B)$ is contained in the fiber of $C\rightarrow T$ over $t$, and $\Supp(\Delta)$ is contained in the fiber of $D\rightarrow T$ over $t$. 
			
			To prove the result, we are free to pass to a stratification and consider each component of $T$, then we may assume $T$ is a smooth variety. 
			Since all pairs in $\Cp$ are normal, after replacing $V$ with its normalization and replacing $C,D$ with their inverse image with reduced structures, we can assume $V$ is normal.
			
			By shrinking $T$, we may assume $V\rightarrow T$ has fiberwise log resolution $\phi:W\rightarrow V$, let $\Sigma$ be the union of the birational transform of $C$ and the reduced exceptional divisor of $\phi$. After a finite base change and possibly shrinking $T$ we can assume that $T$ is smooth, $(W,\Sigma)$ is relatively log smooth over $T$ and $S\rightarrow T$ has irreducible fibers for each stratum $S$ of $(W,\Sigma)$.
			Because for every $(X,B+\Delta)\in \Cp$, $rB$ is integral, then $\mathrm{coeff}(B)$ is in a finite set $\{0,\frac{1}{r},...,\frac{r-1}{r}\}$. By considering different linear combination of irreducible components with coefficients in $\{0,\frac{1}{r},...,\frac{r-1}{r}\}$, we may assume there is a $\mathbb{Q}$-divisor $C^!\leq C$ such that $B=C^!_t$. Let $T'\subset T$ be the set of all closed points such that $t'\in T'$ if and only if there exists a pair in $\Cp$ corresponding to $t'$.
			
			Let $E$ be the reduced exceptional divisor of $W\rightarrow V$, then $E|_{W_t}$ is the reduced exceptional divisor of $W_t\rightarrow X$. 
			Since $(X,B)$ is klt, it is $\epsilon$-lc for some $\epsilon>0$. Let $B'$ be the birational transform of $B$ on $W_t$ plus $(1-\frac{\epsilon}{2})E|_{W_t}$, then $B'\leq \Sigma|_{W_t}$. Let $C_W$ be the strict transform of $C^!$ plus $(1-\frac{\epsilon}{2})E$, then we have $C_W|_{W_t}=B'$. Note $(W,C_W)$ is klt, the coefficients of $C_W$ is $<1$, and $(W,C_W)$ is log smooth over $T$.
			
			Running an MMP on $K_W+C_W$ over $V$ ends with a minimal model $(W',C')$, by the proof of \cite[Lemma 3.16]{Bir22}, $W'\rightarrow V$ is a small contraction. After shrinking $T$, we may assume $W'_s\rightarrow V_s$ is a small contraction for every closed point $s\in T'$. Because every closed point $t'\in T'$ corresponds to a projective pair in $\Cp$, and by assumption, $V_{t'}$ is $\mathbb{Q}$-factorial, then $W'_{t'}\rightarrow V_{t'}$ is an isomorphism. 
			
			Because $W'\rightarrow V$ is a small contraction, then $C'$ is the strict transform of $C^!$. Also because $(V_t,C'_t)\cong (X,B)$ and $W'_t\rightarrow V'_t$ is an isomorphism, then we have $(X,B)\cong (W'_t,C'_t)$. Let $D'$ be the strict transform of $D$ on $W'$, then we have $(X,B+\mathrm{red}(\Delta))\cong (W'_t,C'_t+D'_t)$.
			
			Note $(W,C_W)$ is log smooth over $T$, then $(W,C_W)\rightarrow T$ is a locally stable morphism. Since $W\dashrightarrow W'$ is a $K_W+C_W$-MMP over $V$, it is also a $K_W+C_W$-MMP over $T$. Also because $T$ is smooth, by \cite[Corollary 10]{KNX18}, $(W',C')\rightarrow T$ is a locally stable morphism.
			Since $W'$ is a minimal model, $W'$ is $\mathbb{Q}$-factorial.
			
			Note $T$ is replaced by an open subset, we repeat the argument on the complementary set and get a stratification of $T$. We define $\Cs$ to be the union of the locally closed subset of $T$ such that $T'$ is dense in $\Cs$, let $(\Cx,\Cb)\rightarrow \Cs$ be the locally stable morphism induced by $(W',C')\rightarrow T$ and $\Cd$ the divisor corresponding to $D'$. 
		\end{proof}
	\end{lemma}
	
	\subsection{Canonical bundle formula}
	The following is a simplified version of canonical bundle formula given in \cite[8.5.1]{Kol07}.
	
	Let $(X,B)$ be a sub-lc pair where $B$ is not assumed effective. Let $f:X\rightarrow Z$ be a contraction to a normal variety $Z$ with geometrically connected generic fibers $X_\eta$, where $\eta$ is the generic point of $Z$. Assume that
	\begin{itemize}
		\item $(X_\eta,B_\eta)$ is a Calabi--Yau pair, and
		\item $K_X+B\sim_{\mathbb{Q}} f^*L$ for a $\mathbb{Q}$-Cartier $\mathbb{Q}$-divisor $L$ on $Z$. 
	\end{itemize}
	Let $Z^0\subset Z$ be the largest open set such that $f$ is flat over $Z^0$ with Calabi--Yau fibers and set $Y^0:=f^{-1}Z^0$. Then one can write
	$$K_X+B\sim_{\mathbb{Q}} f^*(K_Z+B_Z+\M_Z)$$
	where $\M_Z$ and $B_Z$ have the following properties:
	\begin{itemize}
		\item $(Z,B_Z+\M_Z)$ is a generalized lc pair,
		\item $\M$ is called the moduli part, which depends only on the generic fiber $(X_\eta,B_\eta)$,
		\item $B_Z$ is called the boundary part, which is supported on $Z\setminus Z^0$, and
		\item suppose $P\subset Z\setminus Z^0$ is a prime divisor, then
		$$\mathrm{coeff}_PB_Z=\sup_E \{1-\frac{1+a(E,X,B)}{\mathrm{mult}_Ef^*P}\}$$
		where the supremum is taken over all divisors over $X$ that dominate $P$.
	\end{itemize}
	Note there exists an open neighborhood $U$ of the generic point of $P$ such that 
	$$\mathrm{coeff}_PB_Z=1-\mathrm{lct}(f^*P_U,X_U,B_U),$$
	where $X_U:=X\times_Z U$, $B_U:=B|_{X_U}$, and $\mathrm{lct}(f^*P_U,X_U,B_U)$ is the largest number $t$ such that $(X_U,B_U+tf^*P_U)$ is lc.
	\begin{lemma}\label{lc place dominates lc place}
		Let $(X,B)$ be an lc pair and $f:X\rightarrow Z$ a contraction such that
		$$K_X+B\sim_{\mathbb{Q}} f^*(K_Z+B_Z+\M_Z).$$
		
		Suppose $P$ is a prime divisor over $X$ whose image on $X$ does not dominates $Z$. Let $Z'\rightarrow Z $ and $X'\rightarrow X$ be two birational morphisms such that we have the following diagram
		$$\xymatrix{
		P\ar@{^{(}->}[r] \ar[d] & X'\ar[d]^{f'} \ar[r]^g & X\ar[d]^f \\
		Q \ar@{^{(}->}[r]& Z'\ar[r]_h & Z,
		}$$
		where 
		\begin{itemize}
			\item $X'\rightarrow Z'$ is a contraction,
			\item $P$ is a prime divisor on $X'$,
			\item $Q$ is a prime divisor on $Z'$, and
			\item $P$ dominates $Q$.
		\end{itemize}
	
		Suppose $a(P,X,B)=-1$ ($\in(-1,0)$, $\in (-1,0]$), then $a(Q,Z,B_Z+\M_Z)=-1$ ($\in(-1,0)$, $\in (-1,0]$).
		\begin{proof}
			Suppose $K_{X'}+B'\sim_{\mathbb{Q}} g^*(K_X+B)$ and $K_{Z'}+B_{Z'}+\M_{Z'}\sim_{\mathbb{Q}} h^*(K_Z+B_Z+\M_Z)$, then $a(P,X,B)=a(P,X',B')$ and $a(Q,Z,B_Z+\M_Z)=-\mathrm{coeff}_QB_{Z'}$. By the construction of $B_{Z'}$, we have
			$$\mathrm{coeff}_QB_{Z'}=\sup_E \{1-\frac{1+a(E,X',B')}{\mathrm{mult}_E(f')^*Q}\}\geq 1-\frac{1+a(P,X',B')}{\mathrm{mult}_P(f')^*Q}.$$
			Because $P$ dominates $Q$ and $Z$ is smooth in codimension 1, then $\mathrm{mult}_P(f')^*Q$ is an integer. So if $a(P,X,B)=-1$ ($\in(-1,0)$, $\in (-1,0]$), then $a(Q,Z,B_Z+\M_Z)=1$ ($\in(-1,0)$, $\in (-1,0]$).
		\end{proof}
	\end{lemma}
	\subsection{Example}
	The following example is provided by Stefano Filipazzi, which shows that in Theorem \ref{Main theorem 1}, we can not bound the degree of the finite cover $\pi:W\rightarrow Z$ without assumptions on the coefficient set of $B$.
	\begin{exa}\label{Stefano's example}
		Consider an elliptic curve $E$ with an $n$-torsion point $P$, let $t$ be a primitive $n$-th root of unity, consider the action of $\mathbb{Z}_n$ on $E\times \mathbb{P}^1$ via $(x,y)\rightarrow (x+P,ty)$. The action has no fixed point, let $X_n$ be the quotient, then the quotient map $E\times \mathbb{P}^1\rightarrow X_n$ is \'etale. It is easy to see that $X_n$ has a smooth $\mathbb{P}^1$ fibration to an elliptic curve $E'$, which is the quotient of $E$ under the action $x\rightarrow x+P$. 
		
		Let $B'_n$ be a $\mathbb{Q}$-divisor on $\mathbb{P}^1$ such that $\mathrm{coeff}(B'_n)<1$, $K_{\mathbb{P}^1}+B'\sim_{\mathbb{Q}} 0$, and $B'_n$ is invariant under the action $y\rightarrow ty$. Then $E\times B'_n$ is invariant under the action of $\mathbb{Z}_n$ and $(E\times \mathbb{P}^1, E\times B'_n)$ is klt, let $B_n$ be the quotient. Since $E\times \mathbb{P}^1\rightarrow X_n$ is \'etale, we have $K_{X_n}+B_n\sim_{\mathbb{Q}} 0$ and $(X_n,B_n)$ is a klt pair.
		It is easy to see that we need a base change of degree $n$ to make the fibration generically trivial.
	\end{exa}
	
	\section{Fibration structures in Calabi--Yau pairs}

	The main result in this section is the following:
	
	\begin{prop}\label{Calabi--Yau pair has a tower of Fano fibration structure which has no very exceptional divisor}
		Suppose $(Y,B_Y)$ is a projective klt Calabi--Yau pair with $B_Y\neq 0$. Then there exist a flop $(X,B)\dashrightarrow (Y,B_Y)$ and a contraction $f:X\rightarrow Z$ such that
		\begin{itemize}
			\item $K_Z\sim_{\mathbb{Q}} 0$,
			\item $f$ factors as a sequence of Fano contractions of relative Picard number 1, and
			\item $f$ has no very exceptional divisor.
		\end{itemize}
		
	\end{prop}
	
	\begin{lemma}\label{generically trivial after a finite base change}
		Suppose $(X,B)$ is a projective klt pair, $f:X\rightarrow Z$ is a contraction such that 
		$$K_{X/Z}+B\sim_{\mathbb{Q}} 0.$$ Then there exists a finite cover $\bar{Z}\rightarrow Z$ such that
		$$(\bar{X},\bar{B})\rightarrow \bar{Z}$$ is a generically trivial fibration, where $\bar{X}$ is the normalization of the main component of $X\times_Z \bar{Z}$ and $\bar{B}$ is the $\mathbb{Q}$-divisor such that $K_{\bar{X}}+\bar{B}$ is equal to the pullback of $K_X+B$.
		\begin{proof}
			Because $K_{X/Z}+B\sim_{\mathbb{Q}} 0$ and $K_X+B$ is $\mathbb{Q}$-Cartier, then $K_X+B\sim_{\mathbb{Q},Z} 0$. By the canonical bundle formula, there exists a generalized pair $(Z,B_Z+\M_Z)$ such that 
			$$K_X+B\sim_{\mathbb{Q}} f^*(K_Z+B_Z+\M_Z).$$
			By assumption we have $B_Z+\M_Z\sim_{\mathbb{Q}} 0$. Because $B\geq 0$, then $B_Z\geq 0$, which means $B_Z\sim_{\mathbb{Q}} \M_Z \sim_{\mathbb{Q}} 0$. Then apply \cite[Theorem 4.7]{Amb05}.
		\end{proof}
	\end{lemma}
	
	\begin{thm}\label{generically trivial is trivial in codimension 1}
		Let $(X,B)$ be a projective klt pair and $f:X\rightarrow Z$ a contraction such that 
		$$K_{X/Z}+B\sim_{\mathbb{Q}} 0,$$
		and every $f$-vertical prime divisor on $X$ dominates a prime divisor on $Z$. 
		
		Suppose $\pi:W\rightarrow Z$ is a finite cover such that $(X_W,B_W)\rightarrow W$ is generically trivial, where $X_W$ is the normalization of the main component of $X\times_Z W$ and $B_W$ is the $\mathbb{Q}$-divisor such that $K_{X_W}+B_W$ is equal to the pull-back of $K_{X}+B$, then there exists a big open subset $U\hookrightarrow Z$ such that
		\begin{itemize}
			\item $X_W$ and $F\times W$ are isomorphic in codimension 1, and
			\item all fibers of $(X,B)\times_Z U\rightarrow U$ over closed points are crepant birationally equivalent.
		\end{itemize}
		\begin{proof}
			First we prove that $f$ has reduced fibers over a big open subset.
			
			By the same argument as in Lemma \ref{generically trivial after a finite base change}, we have
			$B_Z\sim_{\mathbb{Q}} \M_Z\sim_{\mathbb{Q}} 0$, where $B_Z,\M_Z$ are defined by the canonical bundle formula 
			$$K_X+B\sim_{\mathbb{Q}} f^*(K_Z+B_Z+\M_Z).$$
			Because $B_Z=0$, then for any prime divisor $P$ on $Z$, there exists a big open subset $U\hookrightarrow Z$ such that $(X,B+f^*P)$ is lc over $U$. In particular, $f^*P$ is reduced over a big open subset for every prime divisor $P$, then $f$ has reduced fibers over a big open subset. Note $Z$ is normal, then we may assume $U$ is smooth and $P$ is Cartier on $U$, then $f^*P$ is well defined over $U$.
			
			Let $\pi:W\rightarrow Z$ be a finite cover such that $(X_W,B_W)\rightarrow W$ is generically trivial, where $X_W$ is the normalization of the main component of $X\times_Z W$ and $B_W$ is the $\mathbb{Q}$-divisor such that 
			$K_{X_W}+B_W$ is equal to the pull-back of $K_{X}+B$. We denote the finite cover $X_W\rightarrow X$ by $\pi_X$.

			Because $X_W$ is the normalization of the main component of $X\times_Z W$, $f:X\rightarrow Z$ has reduced fibers over a big open subset, and $W\rightarrow Z$ is finite, then $X_W$ is isomorphic to $X\times_Z W$ over a big open subset of $W$ and $f_W:X_W\rightarrow W$ has reduced fibers over a big open subset.

			Define $Y:=X_g\times W$ and $B'_Y:=p_1^*B_g$, where $(X_g,B_g)$ is a general fiber of $f$ and $p_1$ is the projection $X_g\times \bar{Z}\rightarrow X_g$, denote the projection $Y\rightarrow W$ by $f_Y$. Because 
			$K_{X_g}+B_g\sim_{\mathbb{Q}} 0$, then
			$$K_Y+B'_Y\sim_{\mathbb{Q}} f_Y^*K_{\bar{Z}}.$$
			
			By the Hurwitz's formula, there exists an effective divisor $R\geq 0$ such that
			$$K_W=\pi^*K_Z+R.$$
			Define $B_Y:=B'_Y-f_Y^*R$, then 
			$$K_Y+B_Y\sim_{\mathbb{Q}} f_Y^*\pi^* K_Z.$$
			
			Note 
			$K_{X_W}+B_W\sim_{\mathbb{Q}} \pi_X^*(K_X+B)\sim_{\mathbb{Q}} \pi_X^*f^*K_Z$. Because both $\bar{f}:(X_W,B_W)\rightarrow W$ and $f_Y:(Y,B_Y)\rightarrow W$ are generically trivial, then $(X_W,B_W)$ is crepant birationally equivalent to $(Y,B_Y)$.

			Suppose $P$ is a prime divisor on $Z$. Let $H$ be a very ample divisor on $Z$ such that $H-P\sim L\geq 0$ for an effective divisor $L$ and $L$ has no common component with $P$, then $(X_W,B_W+\bar{f}^*\pi^*(P+L))$ is crepant birationally equivalent to $(Y,B_Y+f_Y^*\pi^*(P+L))$. By the Hurwitz formula, there exists a $\mathbb{Q}$-divisor $P_W$ on $W$ such that $K_{W}+P_W\sim \pi^*(K_Z+P+L)$. Let $Q$ be a prime divisor on $W$ which dominates $P$. By \cite[2.42]{Kol13}, $\mathrm{coeff}_QP_W=1$.
			Because $Y\cong {X_g}\times W$ and $K_Y+B_Y+f_Y^*\pi^*(P+L)\sim_{\mathbb{Q}} f_Y^*\pi^*(K_Z+P+L)$, then $f_Y^{-1}Q$ is the only one lc place of $(Y,B_Y+f_Y^*\pi^*(P+L))$ that dominates $Q$.

			Because $f$ has reduced fibers over a big open subset, then $f_W$ has reduced fibers over a big open subset of $W$, in particular, every irreducible component of $f_W^{-1}Q$ is a nklt center of $(X_W,B_W+f_W^*\pi^*(P+L))$ dominating $Q$. 
			Also because $(X_W,B_W+f_W^*\pi^*(P+L))$ is crepant birationally equivalent to $(Y,B_Y+f_Y^*\pi^*(P+L))$, then $f_W^{-1}Q$ and $f_Y^{-1}Q$ are the same divisor on birational models of $X_W$ and $Y$.
			
			By assumption, there is an open subset $U_W\hookrightarrow W$ such that $X_W\times_W U_W\cong {X_g}\times U_W$. 
			Also because every vertical prime divisor of $X_W$ dominates a divisor on $W$ and $f_W^{-1}Q$ and $f_Y^{-1}Q$ are the same divisor for every prime divisor $Q$ on $W$, then $X_W$ is isomorphic in codimension 1 to $X_g\times W$. Furthermore, it is easy to see that $(X_W,B_W)\rightarrow W$ has crepant birationally equivalent fibers over a big open subset of $U_W$, then $(X,B)\rightarrow Z$ has crepant birationally equivalent fibers over a big open subset $U\rightarrow Z$.
		\end{proof}
	\end{thm}
	\begin{lemma}\label{every irreducible component of the boundary dominates the base}
		Let $(X,B)$ be a projective klt pair and $f:X\rightarrow Z$ a contraction such that 
		$$K_{X/Z}+B\sim_{\mathbb{Q}} 0,$$
		and every $f$-vertical prime divisor on $X$ dominates a prime divisor on $Z$. Then every irreducible component of $B$ dominates $Z$ and $f$ has no very exceptional divisor.
		\begin{proof}
			Suppose there exists a prime divisor $P$ on $X$ such that $\mathrm{coeff}_P B=a >0$ and $P$ is vertical over $Z$. By assumption, $P$ dominates a divisor $Q$ on $Z$. 
			
			By the canonical bundle formula we have
			$$K_X+B\sim_{\mathbb{Q}} f^*(K_Z+B_Z+\M_Z),$$
			and by the same argument as in Lemma \ref{generically trivial after a finite base change}, we have
			$B_Z\sim_{\mathbb{Q}} \M_Z\sim_{\mathbb{Q}} 0$.
			Because every log center of $B$ dominates a log center of $(Z,B_Z+\M_Z)$ and $P$ is a log place $(X,B)$, then $Q$ is a log place of $(Z,B_Z+\M_Z)$ and $\mathrm{coeff}_Q(B_Z)>0$. This contradicts with $ B_Z\sim_{\mathbb{Q}} 0$.
			
			By Lemma \ref{generically trivial after a finite base change}, there exists a finite cover $W\rightarrow Z$ such that $(X_W,B_W)\rightarrow W$ is a generically trivial fibration. Also because every $f$-vertical prime divisor on $X$ dominates a prime divisor on $Z$, by Theorem \ref{generically trivial is trivial in codimension 1}, $X_W$ is isomorphic in codimension 1 with ${X_g}\times W$, where ${X_g}$ is a general fiber of $f:X\rightarrow Z$.
			
			Suppose $f$ has very exceptional divisors, since every $f$-vertical prime divisor on $X$ dominates a prime divisor on $Z$, then there exist two prime divisors $Q_1,Q_2$ on $X$ such that $f(Q_1)=f(Q_2)=P$ is a divisor on $Z$. Let $P_W$ be the preimage of $P$ on $W$. Because ${X_g}$ is connected, there is only one divisor on ${X_g}\times W$ that dominates $P_W$, which is ${X_g}\times P_W$. Then preimages of $Q_1$ and $Q_2$ on $X_g\times W$ are both ${X_g}\times P_W$, which is not possible unless $Q_1=Q_2$.
		\end{proof}
	\end{lemma}

	\begin{lemma}\label{make morphisms flat in codimension 1}
		Let $(X,B)$ be a projective klt Calabi--Yau pair and $f:X\rightarrow Z$ a contraction. Suppose there exists a $f$-vertical divisors whose image on $Z$ has codimension $\geq 2$, then we have the following diagram 
		$$\xymatrix{
			(X',B') \ar@{-->}[r] \ar[d] & (X,B)\ar[d] \\
			Z' \ar[r] & Z,
		}$$
		such that
		\begin{itemize}
			\item $X'$ is $\mathbb{Q}$-factorial,
			\item $(X',B')\dashrightarrow (X,B)$ is a flop,
			\item every prime divisor on $X'$ which is vertical over $Z'$ dominates a divisor on $Z'$, and
			\item $(Z',B_{Z'})$ is a klt Calabi--Yau pair for a $\mathbb{Q}$-divisor $B_{Z'}$.
		\end{itemize}
		\begin{proof}
			Let $(Z,B_Z+\M_Z)$ be the generalized pair defined by the canonical bundle formula $K_{X}+B\sim_{\mathbb{Q}} f^*(K_Z+B_Z+\M_Z)$.
			
			Suppose $P_i,i\in \Lambda$ are the prime divisors on $X$ such that $f(P_i)$ has codimension $\geq 2$. Because $a(P_i,X,B)\leq 0$, then $P_i$ dominates a prime divisor $Q_i$ over $Z$ with discrepancy $a(Q_i,Z,B_Z+\M_Z)\leq 0$ according to Lemma \ref{lc place dominates lc place}. 
			Let $g:Z'\rightarrow Z$ be a birational morphism such that $g$ exactly extracts all $Q_i,i\in \Lambda$, let $B'_{Z'}$ be the $\mathbb{Q}$-divisor on $Z'$ such that $K_{Z'}+B'_{Z'}+\M_{Z'}\sim_{\mathbb{Q}} g^*(K_Z+B'_Z+\M_Z)\sim_{\mathbb{Q}} 0$, then $B'_{Z'}\geq 0$. Because $\M_{Z'}$ is $\mathbf{b}$-nef and abundant, there exists an effective $\mathbb{Q}$-divisor $B_{Z'}\sim_{\mathbb{Q}} B'_{Z'}+\M_{Z'}$ such that $(Z',B_{Z'})$ is klt and $K_{Z'}+B_{Z'}\sim_{\mathbb{Q}} 0$.
			
			Let $Y$ be a resolution of indeterminacy of $X\dashrightarrow Z'$, $B_{Y}$ the $\mathbb{Q}$-divisor such that $K_Y+B_Y$ is equal to the pullback of $K_X+B$, and $F$ the sum of exceptional divisors of $Y\rightarrow X$ which are horizontal over $Z$. Because $K_{X_g}+B_g\sim_{\mathbb{Q}} 0$, then $K_{Y_g}+B_{Y_g,>0}+\delta F_g$ has a good minimal model, where $\delta \in(0,1)$ is sufficiently small such that $(Y,B_{Y,>0}+\delta F)$ is klt. By \cite[Theorem 2.12]{HX13}, $(Y,B_{Y,>0}+\delta F)$ has a good minimal model $(X^m,B^m+\delta F^m)$ over $Z'$. 
			
			Because $Y\dashrightarrow X^m$ is a $K_Y+B_{Y,>0}+\delta F\sim_{\mathbb{Q}} B_{Y,<0}+\delta F$-MMP, it only contracts components of $B_{Y,<0}+\delta F$. Because $B_{Y,<0}+\delta F$ is exceptional over $X$, then there is a birational contraction $X^m\dashrightarrow X$. Also because $K_{X_g}+B_g\sim_{\mathbb{Q}} 0$, then $K_{X^m}+B^m+\delta F^m\sim_{\mathbb{Q},Z'} 0$. Since $F$ is horizontal over $Z$ and exceptional over $X$, then $F^m=0$. 
			
			Let $E$ be the sum of $X^m\dashrightarrow X$-exceptional divisors. By the construction of $Z'$, every prime divisor on $X$ which is vertical over $Z'$ dominates a divisor on $Z'$, then $E$ is very exceptional over $Z'$
			
			Because $(Y,B_{Y,>0})$ is klt, then $(X^m,B^m)$ is klt and we can choose $\epsilon>0$ such that $(X^m,B^m+\epsilon E)$ is klt. Since $K_{X^m}+B^m+\epsilon E\sim_{\mathbb{Q}} \epsilon E$, by \cite[Theorem 1.8]{Bir12}, a sequence of $K_{X^m}+B^m+\epsilon E$-MMP over $Z'$ will terminates with a model $X'$ and $X^m\rightarrow X'$ only contracts $E$, then $X$ is isomorphic in codimension 1 with $X'$. Let $B'$ be the pushforward of $B^m$, then $(X,B)\dashrightarrow (X',B')$ is a flop. By construction, every prime divisor on $X'$ which is vertical over $Z'$ dominates a divisor on $Z'$.
		\end{proof}
	\end{lemma}

	\begin{thm}\label{construct the Fano contractions step by step}
		Let $(Y,B_Y)$ be a projective klt Calabi--Yau pair with $B_Y\neq 0$ and $f_Y:Y\rightarrow Z$ a contraction such that every $f_Y$-vertical prime divisor dominates a divisor on $Z$ and $B_Y$ dominates $Z$.
		Then there exist a flop $(X,B)\dashrightarrow (Y,B_Y)$, a birational morphism $X\rightarrow X'$ over $Z$, and a Mori fiber space $X'\rightarrow W$ over $Z$ such that 
		\begin{itemize}
			\item $X\rightarrow X'$ factors as a sequence of divisorial Fano contractions over $Z$ of relative Picard number 1 between $\mathbb{Q}$-factorial varieties,
			\item every prime divisor on $X$ which is vertical over $W$ dominates a divisor on $W$, and
			\item $(W,B_W)$ is a klt Calabi--Yau pair for some $\mathbb{Q}$-divisor $B_W$.
		\end{itemize}
		\begin{proof}
			We prove the result by induction on relative dimension $\mathrm{dim}(Y/Z)$. Suppose the result holds in lower relative dimension.
			
			Because $B_Y$ dominates $Z$, then $-K_Y\sim_{\mathbb{Q}} B_Y \not\sim_{\mathbb{Q},Z} 0$. A sequence of $K_Y$-MMP over $Z$ will terminates to a Mori fiber space. Suppose
			$Y\dashrightarrow W$ is the first step of $K_Y$-MMP. If $Y\dashrightarrow W$ is a divisorial contraction, then we define $X:=X_0=Y$ and $X_1:=W$. If $Y\dashrightarrow W$ is a flip, then we define $X:=X_0=Y$.
			
			Suppose we have a flop $(Y,B_Y)\dashrightarrow (X,B)$ and a sequence of morphisms
			$$X=X_0\rightarrow X_1\rightarrow X_2\rightarrow ...\rightarrow X_r,$$
			where 
			\begin{itemize}
				\item $X_i$ is $\mathbb{Q}$-factorial,
				\item $-K_{X_i}$ is ample over $X_{i+1}$,
				\item $X_i\rightarrow X_{i+1}$ is a divisorial contraction, and 
				\item $\rho(X_i/X_{i+1})=1$ for all $i=0,...,r-1$.
			\end{itemize}
			We construct $X_{r+1}$ as the following.
			
			Let $X_r\dashrightarrow V$ be the next step of the $K_Y$-MMP. If $X_r\rightarrow V$ is a divisorial contraction, then we define $X_{r+1}:=V$.
			
			If $X_r\dashrightarrow V$ is a flip, then by \cite[Proposition 3.7]{BDCS20}, there exists a diagram
			$$\xymatrix{
				X_0 \ar[r] \ar@{-->}[d] & X_1 \ar[r]\ar@{-->}[d] & \cdots \ar[r] & X_{r-1} \ar[r]\ar@{-->}[d] & X_r \ar@{-->}[d] \\ 
				X'_0 \ar[r] & X'_1 \ar[r] & \cdots \ar[r]& X'_{r-1} \ar[r] & X'_r:=V,
			}$$
			where $X_i$ is isomorphic in codimension 1 with $X'_i$ and $X'_i$ is $\mathbb{Q}$-factorial for every $i=0,...,r$. Because $X_i\rightarrow X_{i+1}$ is a divisorial contraction, $-K_{X_i}$ is ample over $X_{i+1}$, and $\rho(X_i/X_{i+1})=1$, we have $X'_i\rightarrow X'_{i+1}$ is a divisorial contraction, $-K_{X'_i}$ is ample over $X'_{i+1}$, and $\rho(X'_i/X'_{i+1})=1$. Then we replace $X_i$ by $X'_i$ for all $i=0,...,r$ and continue.
			
			If $X_r\dashrightarrow V$ is a Mori fiber space. Let $B_r$ be the pushforward of $B$ on $X_r$. Because $-K_{X_r}\sim_{\mathbb{Q}} B_r$ is ample over $V$, then $B$ dominates $V$. We have the following two cases:
			
			Case 1: Every prime divisor on $X_0$ which is vertical over $ V$ dominates a divisor on $V$. Then we define $X':=X_r$ and $W:=V$. It is easy to see that $X\rightarrow X'$ and $X'\rightarrow V$ satisfies the requirements.
			
			Case 2: There exists a prime divisor on $X$ whose image on $V$ has codimension $\geq 2$. Since every prime divisor on $X_0$ dominates a divisor on $Z$, then $\mathrm{dim}(V)>\mathrm{dim}(Z)$. By Lemma \ref{make morphisms flat in codimension 1}, we have a diagram
			$$\xymatrix{
				(X'',B'') \ar@{-->}[r] \ar[d] & (X,B)\ar[d] \\
				V'' \ar[r] & V,
			}$$
			where
			\begin{itemize}
				\item $(X'',B'')\dashrightarrow (X,B)$ is a flop,
				\item $V''\rightarrow V $ is a birational contraction, and
				\item every prime divisor on $X''$ which is vertical over $V''$ dominates a prime divisor on $V''$.
			\end{itemize}
			Since $B$ dominates $V$, then $B''$ dominates $V''$. Because $\mathrm{dim}(X''/V'')<\mathrm{dim}(X/Z)$ and we assume the result in lower relative dimension, then we can apply the theorem on $X''\rightarrow V''$ to get a flop $(Y',B_Y')\dashrightarrow (X'',B'')$, a birational morphism $Y'\rightarrow Y''$ and a Mori fiber space $Y''\rightarrow W'$ over $V''$ such that every prime divisor on $Y'$ which is vertical over $W'$ dominates a divisor on $W'$. 
			
			Note divisorial Fano contractions and Mori fiber spaces over $V''$ are also divisorial Fano contractions and Mori fiber spaces over $Z$. Then we replace $X$ by $Y'$, define $X':=Y''$ and $W:=W'$, it is easy to see that $X\rightarrow X'\rightarrow W$ satisfies the requirements.
		\end{proof}
	\end{thm}
	\begin{thm}\label{Calabi--Yau pair has a tower of Fano fibration structure which is flat in codimension 1}
		Suppose $(Y,B_Y)$ is a projective klt Calabi--Yau pair with $B_Y\neq 0$, then there exist a flop $(X,B)\dashrightarrow (Y,B_Y)$ and a sequence of morphisms
		$$X:=X_0\rightarrow X_1\rightarrow ...\rightarrow X_k:=Z,$$
		where
		\begin{itemize}
			\item $K_Z\sim_{\mathbb{Q}} 0$,
			\item $X_i$ is $\mathbb{Q}$-factorial,
			\item $-K_{X_i}$ is ample over $X_{i+1}$,
			\item $\rho(X_{i}/X_{i+1})=1$, and
			\item every prime divisor on $X$ which is vertical over $Z$ dominates a divisor on $Z$,
		\end{itemize}
		for every $i=0,...,k$.
		\begin{proof}
			Applying Theorem \ref{construct the Fano contractions step by step} on the morphism $Y\rightarrow \mathrm{Spec}\mathbb{C}$, we have a flop $(X,B)\dashrightarrow (Y,B_Y)$, a birational morphism $X\rightarrow X'$ which can be factored as a sequence of divisorial Fano contractions, a Mori fiber space $X'\rightarrow W$ such that every prime divisor on $X$ which is vertical over $W$ dominates a divisor on $W$, and a $\mathbb{Q}$-divisor $B_W$ on $W$ such that $(W,B_W)$ is a klt Calabi--Yau pair. 
			
			Next we construct the sequence of Fano contractions inductively.

			Suppose we have a sequence of Fano contractions
			$$X:=X_0\rightarrow X_1\rightarrow ...\rightarrow X_r\rightarrow V,$$
			where 
			\begin{itemize}
				\item $X_i$ is $\mathbb{Q}$-factorial,
				\item $-K_{X_i}$ is ample over $X_{i+1}$,
				\item $\rho(X_i/X_{i+1})=1$,
				\item $X_r\rightarrow V$ is a Mori fiber space,
				\item every prime divisor on $X$ which is vertical over $V$ dominates a divisor on $V$, and
				\item there is a $\mathbb{Q}$-divisor $B_V$ on $V$ such that $(V,B_V)$ is a klt Calabi--Yau pair,
			\end{itemize}
			for every $i=1,...,r$.
			We construct $X_{r+1},...$ as follows.
			
			If $K_V\sim_{\mathbb{Q}} 0$, then we define $Z:=V$, the sequence $X_0\rightarrow X_1\rightarrow ...\rightarrow X_r\rightarrow Z$ satisfies the requirements.
			
			If $K_V\not \sim_{\mathbb{Q}} 0$, we apply Theorem \ref{construct the Fano contractions step by step} on the morphism $V\rightarrow \mathbb{C}$. Then there exist a flop $(V,B_V)\dashrightarrow (W,B_W)$ and a sequence of morphisms
			$$W:=X_{r+1}\rightarrow X_{r+2}\rightarrow ...\rightarrow X_{l+1},$$
			such that
			\begin{itemize}
				\item $X_i$ is $\mathbb{Q}$-factorial,
				\item $-K_{X_i}$ is ample over $X_{i+1}$,
				\item $\rho(X_i/X_{i+1})=1$,
				\item $X_l\rightarrow X_{l+1}$ is a Mori fiber space,
				\item every prime divisor on $W$ which is vertical over $X_{l+1}$ dominates a divisor on $X_{l+1}$, and
				\item there is a $\mathbb{Q}$-divisor $B_{l+1,X_{l+1}}$ on $X_{l+1}$ such that $(X_{l+1},B_{l+1,X_{l+1}})$ is a klt Calabi--Yau pair,
			\end{itemize}
			for all $i=r+1,...,l+1$.
			
			By \cite[Proposition 3.7]{BDCS20}, we can lift $X_0\rightarrow ...\rightarrow X_r\rightarrow V$ along the small birational contraction $W\dashrightarrow V$ and get a diagram
			$$\xymatrix{
				X_0 \ar[r] \ar@{-->}[d] & X_1 \ar[r]\ar@{-->}[d] & \cdots \ar[r] & X_r \ar[r]\ar@{-->}[d] & V \ar@{-->}[d] \\ 
			X'_0 \ar[r] & X'_1 \ar[r] & \cdots \ar[r]& X'_r \ar[r] & X'_{r+1}=W,
			}$$
			such that every $X_i\rightarrow X'_i$ is an isomorphism in codimension 1 between $\mathbb{Q}$-factorial varieties. It is easy to see that we still have
			\begin{itemize}
				\item $-K_{X'_i}$ is ample over $X'_{i+1}$,
				\item $\rho(X'_i/X'_{i+1})=1$,
				\item $X'_r\rightarrow X'_{r+1}$ is a Mori fiber space, and
				\item every prime divisor on $X_0'$ which is vertical over $X'_{r+1}$ dominates a divisor on $X'_{r+1}$.
			\end{itemize}
			Then we replace the sequence $X:=X_0\rightarrow X_1\rightarrow ...\rightarrow X_r\rightarrow V$ by $X'_0\rightarrow X'_1\rightarrow ...\rightarrow X'_r\rightarrow X'_{r+1}$. To construct the sequence inductively, we only need to show that every prime divisor on $X_0$ which is vertical over $X_{l+1}$ dominates a divisor on $X_{l+1}$.
			
			Suppose $P$ is a prime divisor on $X_0$ which is vertical over $X_{l+1}$, then $P$ is vertical over $X_{r+1}$. By assumption, $P$ dominates a divisor $Q$ on $X_{r+1}$. Then $Q$ is vertical over $X_{l+1}$ and by construction $Q$ dominates a divisor on $X_{r+1}$. Thus every prime divisor on $X_0$ which is vertical over $X_{l+1}$ dominates a divisor on $X_{l+1}$.
		\end{proof}
	\end{thm}

	\begin{proof}[{Proof of Proposition \ref{Calabi--Yau pair has a tower of Fano fibration structure which has no very exceptional divisor}}]
		Let $(X,B)\dashrightarrow (Y,B_Y)$ be the flop and $f:(X,B)\rightarrow Z$ be the morphism constructed in Theorem \ref{Calabi--Yau pair has a tower of Fano fibration structure which is flat in codimension 1}. By Lemma \ref{every irreducible component of the boundary dominates the base}, because every $f$-vertical prime divisor dominates a divisor on $Z$, then every irreducible component of $B$ dominates $Z$ and $f$ has no very exceptional divisor.
	\end{proof}

	\section{Proof of main results}

		\begin{thm}\label{Universal family is potentially trivial after a finite base change}
			Suppose $(\Cx,\Cb)\rightarrow \Cs$ be a family of projective pairs. Then after passing to a stratification of $\Cs$, there exists an \'etale Galois cover $\bar{\Cs}\rightarrow \Cs$, let $(\bar{\Cx},\bar{\Cb})\rightarrow \bar{\Cs}$ be the base change of $(\Cx,\Cb)\rightarrow \Cs$ by $\bar{\Cs}\rightarrow \Cs$, then we have:
			
			Suppose $f:(X,B)\rightarrow S$ is a family of projective pairs and $S\rightarrow \Cs$ is a morphism such that 
			\begin{itemize}
				\item $f$ is isotrivial, and
				\item $(X,B)\rightarrow S$ is isomorphic to the base change of $(\Cx,\Cb)\rightarrow \Cs$.
			\end{itemize}
			Define $\bar{S}:=S\times_\Cs \bar{\Cs}$ and $\bar{f}:(\bar{X},\bar{B})\rightarrow \bar{S}$ be the base change of $f$ via $\bar{S}\rightarrow S$, then $\bar{f}$ is generically trivial.
			\begin{proof}
				Let $\Cx\times \Cs\rightarrow \Cs\times \Cs$ and $\Cs\times \Cx\rightarrow \Cs\times \Cs$ be the two fibrations defined by $\Cx\rightarrow \Cs$.
				Consider the Isom functor
				$\mathbf{Isom}_{\Cs\times \Cs}((\Cx,\Cb)\times \Cs,\Cs\times(\Cx,\Cb))$. By the proof of \cite[\S 1, Theorem 1.10]{Kol96}, $\mathbf{Isom}_{\Cs\times \Cs}((\Cx,\Cb)\times\Cs,\Cs\times(\Cx,\Cb))$ is represented by a locally closed subset
				$$\mathbf{Isom}_{\Cs\times \Cs}((\Cx,\Cb)\times\Cs,\Cs\times(\Cx,\Cb))\subset \mathrm{Hilb}(\Cx\times \Cs \times_{\Cs\times \Cs} \Cs\times \Cx /\Cs\times \Cs)=\mathrm{Hilb}(\Cx\times\Cx /\Cs\times \Cs).$$
				
				Let $\Ci\subset \Cs\times \Cs$ be the image of $\mathbf{Isom}_{\Cs\times \Cs}((\Cx,\Cb)\times\Cs,\Cs\times (\Cx,\Cb))$ on $\Cs\times \Cs$, because $\mathrm{Hilb}(\Cx\times\Cx /\Cs\times \Cs)$ has only countably many components, then by Chevalley’s theorem, $\Ci$ is a disjoint union of countably many locally closed subsets.

				By definition, a closed point $(s,t)\in \Ci$ if and only if $(\Cx_s,\Cb_s)\cong (\Cx,\Cb)\times\Cs|_{\{s,t\}}\cong \Cs\times(\Cx,\Cb)|_{\{s,t\}}\cong (\Cx_t,\Cb_t)$, where $(\Cx,\Cb)\times\Cs|_{\{s,t\}}$ is considered as the fiber of $(\Cx,\Cb)\times \Cs\rightarrow \Cs\times \Cs$ over $\{s,t\}\in \Cs\times \Cs$.
				
				Because the diagonal $\Delta\in \Cs\times \Cs$ is contained in $\Ci$, then the two projections $p_i:\Ci\rightarrow \Cs,i=1,2$ is surjective. Suppose a general fiber $\Ci_g$ of $p_1:\Ci\rightarrow \Cs$ has dimension $d$.
				Let $H\subset \Cs$ be the intersection of $d$ general hypersurfaces on $\Cs$, then $p_2^{-1}(H)$ intersects $\Ci_g$ at countably many points and $p_2^{-1}(H)\cap \Ci$ dominates $\Cs$ via $p_1$.

				Let $\Cx_H\rightarrow H$ be the restriction of $\Cx\rightarrow \Cs$ to $H\subset \Cs$, then every fiber of $\Cx_H\rightarrow H$ is isomorphic only to countably many others. So $\Cx_H\rightarrow H$ is of maximal variation, every fiber of $\Cx_H\rightarrow H$ is isomorphic only to finitely many others. Let $\Ct\subset \Ci\cap p_2^{-1}(H)$ be an irreducible component that dominates $\Cs$ via $p_1$. After replacing $\Cs$ by an open subset and replacing $\Ct$ by its preimage, we may assume that $p_1:\Ct\rightarrow \Cs$ is a finite \'etale cover.

				Consider the Isom functor 
				$\mathbf{Isom}_{\Cs\times H}((\Cx,\Cb)\times H,\Cs\times (\Cx_H,\Cb_H)),$
				which is represented by $\mathrm{Isom}_{\Cs\times H}((\Cx,\Cb)\times H,\Cs\times(\Cx_H,\Cb_H))$. By definition, we have $\mathrm{Isom}_{\Cs\times H}((\Cx,\Cb)\times H,\Cs\times(\Cx_H,\Cb_H))=\mathrm{Isom}_{\Cs\times \Cs}((\Cx,\Cb)\times \Cs,\Cs\times(\Cx,\Cb))\times _{\Cs\times \Cs} \Cs\times H$. It is easy to see that $\Ct$ is an irreducible component of the image of $\mathrm{Isom}_{\Cs\times H}((\Cx,\Cb)\times H,\Cs\times(\Cx_H,\Cb_H))$ on $\Cs\times H$.
				Let $\Cv\subset \mathrm{Isom}_{\Cs\times H}((\Cx,\Cb)\times H,\Cs\times (\Cx_H, \Cb_H))$ be a locally closed subset such that $\Cv\rightarrow \Ct$ is a quasi-finite dominant map. After replacing $\Cs$ by an open subset, we may assume that $\Cv\rightarrow \Ct$ is a finite \'etale cover.

				Suppose $f:(X_U,B_U)\rightarrow U$ is a family of projective pairs and there exists a morphism $\phi:U\rightarrow \Cs$ such that 
				\begin{itemize}
					\item $f$ is isotrivial, and
					\item $(X_U,B_U)\rightarrow U$ is isomorphic to the base change of $(\Cx,\Cb)\rightarrow \Cs$ via $\phi$.
				\end{itemize}
				Define $T:=U\times_\Cs \Ct,V:=U\times_\Cs \Cv, (X_V,B_V):=(X_U,B_U)\times_U V$ and $(X_T,B_T):=(X_U,B_U)\times_U T$. Because $f$ is isotrivial and every fiber of $(\Cx_H,\Cb_H)\rightarrow H$ is isomorphic to only finitely many others, then the image of $V$ on $H$ via $V\rightarrow \Cv \xrightarrow{p_2} H$ is a closed point. By the definition of Isom functor, $(X_V,B_V)\rightarrow V$ is both isomorphic to the base change of $(\Cx,\Cb)\times H\rightarrow \Cs\times H$ and the base change of $\Cs\times (\Cx_H,\Cb_H)\rightarrow \Cs\times H$ via $V\rightarrow \Cv\rightarrow \Cs\times H$. Because the image of $V\rightarrow \Cv \xrightarrow{p_2} H$ is a closed point, then the base change of $\Cs\times (\Cx_H,\Cb_H)\rightarrow \Cs\times H$ via $V\rightarrow \Cv\rightarrow \Cs\times H$ is a trivial fibration, which is equal to say $(X_V,B_V)\rightarrow V$ is a trivial fibration.
				
				Note we replace $\Cs$ by an open subset, so we repeat the same argument on the complementary set to get a stratification of $\Cs$, and on each component there exists a finite \'etale cover $\Cv\rightarrow \Cs$. Then we define $\bar{\Cs}\rightarrow \Cs$ to be the Galois closure of $\Cv\rightarrow \Cs$ over each irreducible component of $\Cs$. Because $(X_V,B_V)\rightarrow V$ is a trivial fibration and $U\times_{\Cs}\bar{\Cs}\rightarrow U$ factors through $V$, then $(X_U,B_U)\times _{\Cs} \bar{\Cs}\rightarrow U\times_{\Cs}\bar{\Cs}$ is a trivial fibration, which satisfies the requirement.
			\end{proof}
		\end{thm}
		\begin{thm}\label{towers of e-log Calabi--Yau fibrations are bounded}
			Fix $d,k\in \mathbb{N},\epsilon \in \mathbb{Q}^{>0}$, then there exists $v\in \mathbb{N}$ depending only on $d,k,\epsilon$ satisfying the following.
			
			Suppose $(X,B+\M_X)$ be a projective $d$-dimensional generalized $\epsilon$-lc pair with a tower of contractions
			$$X:=X_0\xrightarrow{p_0}X_1\xrightarrow{p_1}...\xrightarrow{p_{k-1}}X_k=Z,$$
			such that $K_X+B+\M_X\sim_{\mathbb{Q}} 0$ and $-K_{X_i}$ is ample over $X_{i+1}$ for every $0\leq i< k$.
			Then there exists a divisor $A$ on $X$ over $Z$ such that $A_g$ is very ample and $\mathrm{vol}(A_g)\leq v$, where $A_g:=A|_{X_g}$ and $X_g$ is a general fiber of $X\rightarrow Z$.
			\begin{proof}
				Denote the contractions $X\rightarrow X_i$ by $e_i$ and $X_i\rightarrow Z$ by $f_i$. Let $g\in Z$ be a general point, denote the fiber of $f_i$ over $g$ by $F_i$. 
				
				Because $K_X+B+\M_X\sim_{\mathbb{Q},X_i} 0$.
				By the canonical bundle formula, there are generalized pairs $(X_i,B_i+\M_{i,X_i})$ such that
				$$K_X+B+\M_X\sim_{\mathbb{Q}}e_i^*(K_{X_i}+B_i+\M_{i,X_i}).$$
				
				By \cite[Theorem 9.3]{Bir23}, there exists $\delta_i\in (0,1)$ depending only on $d,\epsilon$ such that $(X_i,B_i+\M_{i,X_i})$ is generalized $\delta_i$-lc. For simplicity of notation, we replace $\epsilon$ by $\min\{\epsilon,\delta_i,i=1,...,k\}$ so that $(X_i,B_i+\M_{i,X_i})$ is generalized $\epsilon$-lc for every $0\leq i\leq k$.
				
				We prove the result by induction on $k$.
				
				When $k=1$. 
				Because $-K_{X}$ is ample over $Z=X_1$, then $-K_{X_g}$ is ample. Also because $(X,B+\M_X)$ is generalized $\epsilon$-lc, $X_g$ is a $\epsilon$-lc Fano variety. By the Birkar-BAB Theorem (see \cite{Bir21a}), $X_g$ is in a bounded family. By boundedness, there exists $l_1,v_1\geq 0$ depending only on $d,\epsilon$ such that $-l_1K_{X_g}$ is very ample and $\mathrm{vol}(-l_1K_{X_g})\leq v_1$, then we let $A:=-l_1K_X$.
				
				Suppose the statement is true for length $=k-1$. Applying it on $X_1\rightarrow ...\rightarrow Z$, then there exists a divisor $H$ on $X_1$ such that $H|_{F_1}$ is very ample and $\mathrm{vol}(H|_{F_1}) \leq v_{r-1}$. Because $-K_X$ is ample over $X_1$, then $-K_{F_0}$ is ample over $F_1$. By \cite[Definition 2.2]{Bir22}, $(F_0,(B+\M_X)|_{F_0})\rightarrow F_1$ is a generalized $(d,v_{r-1},\epsilon)$-Fano type fibration, then by \cite[Theorem 2.3]{Bir22}, $X$ is in a bounded family. By \cite[Lemma 4.4]{Bir22}, there exist $m_r$ and $l_r$ depending only on $d,v_{r-1},\epsilon$ such that 
				$$m_r(l_re_1^* H|_{F_0}-K_{F_0})=m_r(l_re_0^*H-K_X)|_{F_0}$$
				is very ample. By \cite[Proposition 4.8]{Bir22}, there exists $v_r$ depending only on $d,v_{r-1},\epsilon$ such that $\mathrm{vol}(m_r(l_re_1^* H|_{F_0}-K_{F_0}))\leq v_r$. Then we define $A:=m_r(l_re_0^*H-K_X)$.
			\end{proof}
		\end{thm}

			\begin{proof}[Proof of Theorem \ref{Main theorem 1}]
				By Proposition \ref{Calabi--Yau pair has a tower of Fano fibration structure which has no very exceptional divisor}, there exist a flop $(X,B)\dashrightarrow (Y,B_Y)$ and a contraction $f:X\rightarrow Z$ such that
				\begin{itemize}
					\item $K_Z\sim_{\mathbb{Q}} 0$,
					\item $f$ factors as a sequence of Fano contractions, and 
					\item $f$ has no very exceptional divisor.
				\end{itemize}

				Suppose $f$ factors as 
				$$X:=X_0\rightarrow ...\rightarrow X_m=Z.$$
				Let $g\in Z$ be a general point and $X_{i,g}$ be fiber of $X_i\rightarrow Z$ over $g$, then $-K_{X_{i,g}}$ is ample over $X_{i+1,g}$.
				
				Note $(X_g,B_g)$ is $\epsilon$-lc, then by \cite[Theorem 1.6]{Bir23}, $X_g$ is bounded in codimension 1. There exists a contraction $\Cv\rightarrow \Ct$ over a scheme $T$ of finite type depending only on $d,\epsilon$ such that $X_g$ is isomorphic in codimension 1 with a fiber $\Cv_t$ of $\Cv\rightarrow \Ct$. By taking a $\mathbb{Q}$-factorization, we may assume $\Cv_t$ is $\mathbb{Q}$-factorial, note the boundedness is kept by \cite[Theorem 1.2]{Bir22}. After passing to a stratification of $\Ct$ and taking a resolution of the generic fiber of $\Cv\rightarrow \Ct$, we may assume $\Cv'\rightarrow \Ct$ is a fiberwise resolution of $\Cv\rightarrow \Ct$. Since smooth morphisms are locally products in the complex topology, then $\mathrm{dim}_{\mathbb{R}}H^2(\Cv'_t,\mathbb{R})$ is bounded by some natural number $l$ depending only on $\Cv\rightarrow \Ct$, hence depending only on $d,\epsilon$. Since the N\'eron-Severi group $\mathrm{N}^1(\Cv'_t)$ is embedded in $H^2(\Cv'_t,\mathbb{R})$ as a vector space, then 
				$$\rho(\Cv'_t)=\mathrm{dim}\mathrm{N}(\Cv'_t)\leq \mathrm{dim}_{\mathbb{R}}H^2(\Cv'_t,\mathbb{R})\leq l.$$
				Because $\Cv'_t\dashrightarrow X_g$ is a birational contraction, then $\rho(X_g)\leq \rho(\Cv'_t)\leq l$.
				
				Since $\rho(X_{i,g}/X_{i+1,g})\geq 1$ and $\rho(X_g)\leq l$, we have $m\leq l$, then length of the sequence of Fano contractions is $\leq l$.
				
				By Theorem \ref{towers of e-log Calabi--Yau fibrations are bounded}, there exists a divisor $A$ on $X$ such that $A_g$ is very ample and $A_g^{\mathrm{dim}(X_g)}\leq v$, where $v>0$ depends only on $d,l,\epsilon$. In particular, $X_g$ is in a bounded family depending only on $d,v$. Because $K_{X_g}+B_g\sim_{\mathbb{Q}} 0$, by boundedness of $X_g$, there exists $u>0$ depending only on $d,v$ such that $B_g. A_g^{\mathrm{dim}(X_g)}\leq u$.

				Fix $l'\in \mathbb{N}$ such that $l'B$ is an integral divisor.
				By Lemma \ref{a general fiber has a very ample divisor implies bounded moduli space}, there exist a family of projective pairs $(\Cx,\Cb)\rightarrow \Cs$ depending only on $d,l',v$, an open subset $U\hookrightarrow Z$ and a morphism $U\rightarrow \Cs$ such that $(X_U,B_U):=(X,B)\times_Z U\rightarrow U$ is isomorphic to the base change of $(\Cx,\Cb)\rightarrow \Cs$ by $U\rightarrow \Cs$.
				
				Because $K_{X/Z}+B\sim_{\mathbb{Q}} 0$, by Lemma \ref{generically trivial after a finite base change}, there exists a finite cover $\bar{Z}\rightarrow Z$ such that $(\bar{X},\bar{B})\rightarrow \bar{Z}$ is generically trivial, where $\bar{X}$ is the normalization of the main component of $X\times _Z \bar{Z}$ and $\bar{B}$ is the $\mathbb{Q}$-divisor such that $K_{\bar{X}/\bar{Z}}+\bar{B}$ is equal to the pull-back of $K_{X/Z}+B$. After shrinking $U$, we may assume that $(\bar{X},\bar{B})$ is isomorphic with $ (X,B)\times_Z \bar{Z}$ over $U$, then $(X_U,B_U)\rightarrow U$ is isotrivial.
				
				After passing to a stratification of $\Cs$, we may let $\bar{\Cs}\rightarrow \Cs$ be the \'etale Galois cover and $(\bar{\Cx},\bar{\Cb})\rightarrow \bar{\Cs}$ be the morphism defined in Theorem \ref{Universal family is potentially trivial after a finite base change}, let $r$ be the degree of $\bar{\Cs}\rightarrow \Cs$. We also replace $U$ by an open subset such that $U\rightarrow \Cs$ is still a morphism. Define $U_W:=U\times_\Cs \bar{\Cs}$ and $(X_{U_W},B_{U_W}):=(X,B)\times _U U_W$, then $U_W\rightarrow U$ is an \'etale Galois cover and $(X_{U_W},B_{U_W})\rightarrow U_W$ is a trivial fibration. 
				
				Let $W$ be the closure of $U_W$ such that $U_W\rightarrow U$ extends to a morphisms $W\rightarrow Z$. By Stein factorization we may assume $W$ is normal and $W\rightarrow Z$ is a finite cover, then $W\rightarrow Z$ is an \'etale Galois cover. Let $X_W$ be the normalization of the main component of $X\times_Z W$ and $B_W$ be the closure of $B_{U_W}$ in $X_W$, then $(X_W,B_W)\rightarrow W$ is generically trivial.
				
				Let $G$ be the Galois group $\mathrm{Gal}(W/Z)$, $H$ be the subgroup generated by the ramified group $I(P)$ for every prime divisor $P$ on $Z$. We replace $W$ by $W/ H$ and replace $(X_W,B_W)$ accordingly. Then by the proof of \cite[Theorem 4.7]{Amb05}, $(X_W,B_W)\rightarrow W$ is generically trivial. By Theorem \ref{generically trivial is trivial in codimension 1}, there exists a big open subset $U\hookrightarrow Z$ such that
				\begin{itemize}
					\item $X_W$ and $F\times W$ are isomorphic in codimension 1, and
					\item $(X,B)\times_Z U\rightarrow U$ has crepant birationally equivalent fibers.
				\end{itemize}
				
				If $\mathrm{coeff}(B)\subset \Ci$ for a fixed DCC set $\Ci\subset \mathbb{Q}\cap (0,1)$, then by \cite{HMX14}, there exists $l'\in \mathbb{N}$ depending only on $d,\Ci$ such that $l'B$ is an integral divisor. Then the family of projective pairs $(\Cx,\Cb)\rightarrow \Cs$ depends only on $d,v,\Ci$. Also because $v$ depends only on $l,\epsilon$, we have $\mathrm{deg}(W\rightarrow Z)\leq \mathrm{deg}(\bar{\Cs}\rightarrow \Cs)=r$ depends only on $d,l,\epsilon,\Ci$.
			\end{proof}

		\begin{thm}\label{Kodaira dimensions are the same globally and fiberwisely}
			Let $(X,B+\Delta)$ be a $\mathbb{Q}$-factorial klt Calabi--Yau pair and $f:X\rightarrow Z$ a contraction such that
			\begin{itemize}
				\item $B,\Delta$ are effective $\mathbb{Q}$-divisors,
				\item $K_{X/Z}+B+\Delta \sim_{\mathbb{Q}} 0$, and
				\item $f$ has no very exceptional divisors.
			\end{itemize}
			Then we have
			$$\kappa(B)=\kappa(B_g)\text{ and }\nu(B)= \nu(B_g),$$
			where $X_g$ is a general fiber of $f$ and $\Delta_g:=\Delta|_{X_g},B_g:=B|_{X_g}$.
			\begin{proof}
				By Lemma \ref{generically trivial after a finite base change} and Theorem \ref{generically trivial is trivial in codimension 1}, there exists a finite cover $\pi:W\rightarrow Z$ such that $(X_W,B_W+\Delta_W)\rightarrow W$ is generically trivial and 
				$X_W$ is isomorphic in codimension 1 with $X_g\times W$, where $X_W$ is the normalization of the main component of $X\times_Z W$, $B_W:=\pi_X^* B$ and $\Delta_W:=\pi_X^* \Delta$. Suppose $R$ is the ramified divisor of $\pi$, then by Hurwitz's formula we have
				$$K_W=\pi^*K_Z+R.$$
				Let $p_1:X_g\times W\rightarrow X_g$ and $p_2:X_g\times W\rightarrow W$ be the projection, then $$K_{X_g\times W}+p_1^*\Delta_g -p_2^*R=p_1^*(K_{X_g}+\Delta_g)+p_2^*(K_W-R)\sim_{\mathbb{Q}} p_1^*(K_{X_g}+\Delta_g)\sim_{\mathbb{Q}} -p_1^*B_g.$$
				Note $X$ is $\mathbb{Q}$-factorial, then $B_g$ is $\mathbb{Q}$-Cartier. Also because $p_1$ is surjection, then by \cite[\S II, Lemma 3.11]{Nak04} and \cite[\S V, Proposition 2.7]{Nak04}, we have
				$$\kappa(p_1^*B_g)=\kappa(B_g)\text{ and }\nu(p_1^*B_g)=\nu(B_g).$$

				Let $f_W$ be the natural morphism $X_W\rightarrow W$. By Theorem \ref{generically trivial is trivial in codimension 1}, $f$ has reduced fibers over a big open subset. Also because $f$ has no very exceptional divisor, then the ramified divisor of $X_W\rightarrow W$ is $f_W^* R$ and we have
				$K_{X_W}-f_W^*R=\pi_X^* K_X$, where $\pi_X$ is the finite cover $X_W\rightarrow X$. Then 
				$$K_{X_W}+\Delta_W-f_W^*R=\pi^*(K_X+\Delta)\sim_{\mathbb{Q}} -B_W.$$
				Similarly we have 
				$$\kappa(B)=\kappa(B_W)\text{ and }\nu(B)=\nu(B_W).$$
				
				Note
				$$K_{X_g\times W}  -p_2^*R+p_1^*(B_g+\Delta_g)\sim_{\mathbb{Q}} 0,$$
				because $f_W^*R$ is the strict transform of $p_2^*R$ on $X_W$ and $p_1^*(\Delta_g+B_g)$ is trivial over $W$, then the strict transform of $B_W+\Delta_W$ on $X_g\times W$ is $p_1^*(B_g+\Delta_g)$.

				Since $X_g\times W$ is isomorphic in codimension 1 with $X_W$ and the strict transform of $p_1^*B_g$ on $X_W$ is $B_W$. There exists an isomorphism $H^0(X_g\times W, \Co_{X_g\times W}(mp_1^*B_g))\cong H^0(X_W,\Co_{X_W}(mB_W))$ for every $m\in \mathbb{N}$,
				 then we have 
				$$\kappa(p_1^*B_g)=\kappa(B_W).$$
				Similarly for any ample divisor $A'$ on $X_g\times W$, its strict transform $A$ on $X_W$ is big and $$H^0(X_g\times W, \Co_{X_g\times W}(mp_1^*B_g+A'))\cong H^0(X_W,\Co_{X_W}(mB_W+A)),$$
				for all $m\in \mathbb{N}$, then we have
				$$\nu(p_1^*B_g)=\nu(B_W).$$
				Also because $\kappa(B)=\kappa(B_W)$, $\nu(B)=\nu(B_W)$, $\kappa(p_1^*B_g)=\kappa(B_g)$, and $\nu(p_1^*B_g)=\nu(B_g),$ then we have
				$$\kappa(B)=\kappa(B_g)\text{ and }\nu(B)=\nu(B_g).$$
			\end{proof}
		\end{thm}
		
		\begin{proof}[Proof of Theorem \ref{Main corollary, decomposition for CY}]
			We prove the result by induction on dimension. Suppose $\mathrm{dim}(X)=d$ and the Theorem holds in dimension $d-1$.
			
			If $B_Y=0$ and $(Y,0)$ is canonical, we just let $Z:=Y$. So we may assume $B_Y\neq 0$ or $(Y,B_Y)$ is not canonical.
			
			Let $(Y',B'_Y+E'_Y)\rightarrow (Y,B_Y)$ be a projective birational morphism which exactly extracts every log place of $(Y,B_Y)$, where $B'_Y$ is the strict transform of $B_Y$ and $E'_Y$ is effective and exceptional over $Y$. It is easy to see that $B'_Y+E'_Y\neq 0$, $E'_Y$ has the same support with the sum of all exceptional divisors over $Y$, and $\nu(E'_Y)=\kappa(E'_Y)=0$.
			
			Let $(Y'',B''_Y+E''_Y)\dashrightarrow (Y',B'_Y+E'_Y)$ be the flop defined in Theorem \ref{Main theorem 1}, we replace $(Y'',B''_Y+E''_Y)$ by $(Y',B'_Y+E'_Y)$. Note we still have $\nu(E'_Y)=\kappa(E'_Y)=0$. Let $h:Y'\rightarrow Z_1$ be the contraction defined in Theorem \ref{Main theorem 1} and $Y'_g$ a general fiber of $h$.
			
			Choose $\delta\in (0,1)$ sufficiently small such that $(Y',B'_Y+(1+\delta)E'_Y)$ is klt. By Theorem \ref{Kodaira dimensions are the same globally and fiberwisely}, because $\nu(E'_Y)=\kappa(E'_Y)=0$, we have $\nu(K_{Y'_g}+(B'_Y+(1+\delta)E'_Y)|_{Y'_g})=0$, then by \cite{Gon11}, $(Y',B'_Y+(1+\delta)E'_Y)$ has a good minimal model $(Y^m,B^m_Y+(1+\delta)E^m_Y)$ over $Z_1$, where $B^m_Y,E^m_Y$ are the pushforward of $B'_Y,E'_Y$. Since $K_{Y'}+B'_Y+(1+\delta)E'_Y\sim_{\mathbb{Q}} \delta E'_Y$ and every irreducible component of $E'_Y$ dominates $Z_1$, then $Y'\dashrightarrow Y^m$ exactly contracts every components of $E'_Y$ and the pushforward of $E'_Y$ on $Y^m$ is $0$. Because $E'_Y$ has the same support with the sum of all exceptional divisors over $Y$, then $(Y^m,B^m_Y)\dashrightarrow (Y,B_Y)$ is a flop. Because $h$ has reduced divisors over codimension 1 points of $Z_1$ and no very exceptional divisor, then $h^m:Y^m\rightarrow Z_1$ also has reduced divisors over codimension 1 points of $Z_1$ and no very exceptional divisor.
			
			By induction on dimension there exist a flop $Z'_1\dashrightarrow Z_1$ and a contraction $g:Z'\rightarrow Z$ such that $Z$ is a canonical Calabi--Yau variety and $g$ has reduced divisors over codimension 1 points of $Z$ and no very exceptional divisor. By \cite[Proposition 3.7]{BDCS20}, there exist a flop $(X,B)\rightarrow (Y^m,B^m_Y)$ and the following commutative diagram
			$$\xymatrix{
			(X,B)\ar[d] \ar@{-->}[r] & (Y^m,B^m_Y) \ar[d] \\
			Z'_1  \ar@{-->}[r] \ar[d] & Z_1\\
			Z.
			}$$
			Because both $X\rightarrow Z'_1$ and $Z'_1\rightarrow Z$ has reduced divisors over codimension 1 points and no very exceptional divisor, then $f:=g\circ h:X\rightarrow Z$ has reduced divisors over codimension 1 points of $Z$ and no very exceptional divisor. Then the result follows from \cite[Theorem 4.7]{Amb05} and Theorem \ref{generically trivial is trivial in codimension 1}.
		\end{proof}
		
		The following result generalize Theorem \ref{Main theorem 2} to Calabi--Yau pair case.
		
		\begin{thm}\label{Main theorem 2 with boundary}
			Fix $d,r\in \mathbb{N}$, then there exists $l\in \mathbb{N}$ depending only on $d,r$ such that:

			Let $(X,B)$ be a $d$-dimensional projective klt variety such that $K_X+B\sim_{\mathbb{Q}} 0$ and $rB$ is integral, $h:(Y,B_Y)\rightarrow X$ be a terminalization of $(X,B)$, then $lB_Y$ is an integral divisor. 
			\begin{proof}
				Suppose $B_Y=h_*^{-1}B+E_Y$. Because $(Y,B_Y)\rightarrow (X,B)$ is a terminalization of $X$, then $E_Y$ is effective and exceptional over $X$ and $\kappa(-(K_Y+h_*^{-1}B))=\kappa(E_Y)=\nu(-(K_Y+h_*^{-1}B))=\nu(E_Y)=0$.
				
				Because $(X,B)$ is a $d$-dimensional Calabi--Yau pair and $rB$ is integral, by \cite{HMX14}, there exists $\epsilon\in (0,1)$ depending only on $d,r$ such that $(X,B)$ is $\epsilon$-lc, then $(Y,B_Y)$ is also $\epsilon$-lc. 
				
				Let $(X',B')\dashrightarrow (Y,B_Y)$ be the flop given in Theorem \ref{Main theorem 1}. To show that $lB_Y$ is integral, we only need to show that $lB'$ is integral, then we may replace $(Y,B_Y)$ with $(X',B')$. Note we still have $\kappa(-(K_Y+h_*^{-1}B))=\kappa(E_Y)=\nu(-(K_Y+h_*^{-1}B))=\nu(E_Y)=0$.
				
				By Theorem \ref{Main theorem 1}, there exist a contraction $f:Y\rightarrow Z$ and a finite cover $\pi:W\rightarrow Z$ such that
				\begin{itemize}
					\item $\mathrm{deg}(\pi)=r$, where $r\in \mathbb{N}$ depends only on $d$,
					\item $f$ has no very exceptional divisor, and
					\item $Y_W$ is isomorphic in codimension 1 with $Y_g\times W$, where $Y_W$ is the normalization of the main component of $Y\times_Z W$ and $Y_g$ is a general fiber of $f$.
				\end{itemize}
				
				By Theorem \ref{Kodaira dimensions are the same globally and fiberwisely}, we have 
				$$\kappa(E_Y)=\kappa(E_Y)=0\text{ and }\nu(E_Y)= \nu(E_Y)=0.$$
				Let $\delta\in (0,1)$ be a rational number such that $(Y,h_*^{-1}B+(1+\delta)E_Y)$ is klt. Because $\nu(-(K_{Y_g}+(h_*^{-1}B)|_{Y_g}))=\nu(K_{Y_g}+(h_*^{-1}B)|_{Y_g}+(1+\delta)E_{Y_g})=0$, by \cite{Gon11}, $(Y_g,(h_* ^{-1}B)|_{Y_g}+(1+\delta)E_{Y_g})$ has a good minimal model. Also because $(Y,h_*^{-1}B+(1+\delta)E_Y)$ is klt, by \cite[Theorem 2.12]{HX13}, $(Y,h_*^{-1}B+(1+\delta)E_Y)$ has a good minimal model over $Z$, we denote it by $Y\dashrightarrow Y^m$. Because $K_Y+h_*^{-1}B+(1+\delta)E_Y\sim_{\mathbb{Q}} \delta E_Y$ and $\nu(E_Y)=0$, then $E_Y$ is contracted by $Y\dashrightarrow Y^m$ and $K_{Y^m}+B^m\sim_{\mathbb{Q},Z} 0$, where $B^m:=(h_m)_*^{-1}B$ and $h_m$ is the natural birational map $Y^m\dashrightarrow X$.
				
				Because a general fiber $Y_g$ factors as a sequence of Fano contraction, $Y_g$ is rationally connected, then $Y^m_g$ is also rationally connected. Since $K_{Y^m_g}+B^m_g\sim_{\mathbb{Q}} 0$, by \cite{HMX14}, there exists $\epsilon'\in (0,1)$ depending only on $d$ such that $(Y^m_g,B^m_g)$ is $\epsilon'$-lc. By \cite[Theorem 1.6]{Bir23}, $Y^m_g$ is bounded in codimension 1. Because $rB^m_g$ is integral and $K_{Y^m_g}+B^m_g\sim_{\mathbb{Q}} 0$, then $(Y^m_g,B^m_g)$ is log bounded in codimension 1.

				Let $\Cp$ be the corresponding log bounded set of rationally connected $\epsilon'$-lc Calabi--Yau varieties. After taking $\mathbb{Q}$-factorization, by \cite[Theorem 1.2]{Bir22}, we may also assume varieties in $\Cp$ are $\mathbb{Q}$-factorial.
				By Lemma \ref{first step of relative MMP}, there exist a locally stable morphism $(\Cx,\Cb)\rightarrow \Cs$ and a dense subset $\Cs'\subset \Cs$ such that there exists a closed point $s_0\in \Cs$ such that $(Y^m_g,B^m_g)$ is isomorphic in codimension 1 with $(\Cx_{s_0},\Cb_{s_0})$ and a closed point $s'\in \Cs'$ if and only if there exists $(W,C)\in \Cp$ together with an isomorphism $(W,C)\cong (\Cx_{s'},\Cb_{s'})$. Because $K_{Y^m_g}+B^m_g\sim_{\mathbb{Q}} 0$, then $K_{\Cx_{s_0}}+\Cb_{s_0}\sim_{\mathbb{Q}} 0$. Thus the fibers are Calabi--Yau over a dense subset, also because $K_{\Cx}+\Cb$ is $\mathbb{Q}$-Cartier, after shrinking $\Cs$, we have $K_{\Cx}+\Cb\sim_{\mathbb{Q},\Cs} 0$.
				
				Let $l\in \mathbb{N}$ such that $l(K_{\Cx}+\Cb)\sim_{\Cs} 0$, then $l(K_{\Cx_s}+\Cb_s)\sim 0$ for every $s\in \Cs$. Also because $(Y^m_g,B^m_g)$ is isomorphic in codimension 1 with $(\Cx_{s_0},\Cb_{s_0})$, then $l(K_{Y^m_g}+B^m_g)\sim 0$. Because $(Y_g,B_{Y_g})$ is crepant birationally equivalent with $(Y^m_g,B^m_g)$, then $l(K_{Y_g}+B_{Y_g})\sim 0$. In particular, we have $lB_{Y_g}$ is integral. Because $B_Y$ is horizontal over $Z$, then $lB_Y$ is integral.
			\end{proof}
		\end{thm}
		
		\begin{thm}\label{Main corollary 1 with boundary}
			Fix $d,r\in \mathbb{N}$. Assume Conjecture \ref{index conjecture} in dimension $d-1$,
			then there exists $l\in \mathbb{N}$ depending only on $d,r$ such that if $(Y,B_Y)$ is a $d$-dimensional non-canonical Calabi--Yau pair and $rB_Y$ is integral, then $l(K_Y+B_Y)\sim 0$.
			\begin{proof}
				By induction on dimension, we may assume the theorem holds in dimension $d-1$. 
				
				Because $rB_Y$ is integral, by \cite{HMX14}, there exists $\epsilon \in (0,1)$ depends only on $d,r$ such that $(Y,B_Y)$ is an $\epsilon$-lc Calabi--Yau pair.
				If $X$ is rationally connected, then by \cite[Theorem 1.6]{Bir23}, $(Y,B_Y)$ is log bounded in codimension 1, thus there exists $l$ depending only on $d,r,\epsilon$ such that $l(K_Y+B_Y)\sim 0$. So we may assume $X$ is not rationally connected.
				
				By Theorem \ref{Main theorem 1}, there exist $m\in \mathbb{N},v\in \mathbb{Q}^{>0}$, a flop $(X,B)\dashrightarrow (Y,B_Y)$, a contraction $f:X\rightarrow Z$, a divisor $A$ on $X$, and a finite cover $\pi:W\rightarrow Z$ such that
				\begin{itemize}
					\item $K_Z\sim_{\mathbb{Q}} 0$,	
					\item $f$ factors as a sequence of Fano contractions with length $\leq m$,
					\item $A_g:=A|_{X_g}$ is very ample with $\mathrm{vol}(A_g)\leq v$, and
					\item $X_W$ is isomorphic in codimension 1 with $X_g\times W$, where $X_W$ is the normalization of the main component of $X\times _Z W$.
				\end{itemize}
				Because $X$ is not rationally connected, then $Z$ is not a closed point and $\mathrm{dim}(X_g)\leq d-1$.
				
				Because $A_g$ is very ample and $\mathrm{vol}(A_g)\leq v$, then $X_g$ is bounded and $-K_{X_g}.(A_g)^{d-1}$ is bounded from above. Since $K_{X_g}+B_g\sim_{\mathbb{Q}}0$, then $B_g. (A_g)^{d-1}$ is bounded from above. Also because $rB_g$ is an integral divisor, then $(X_g,B_g)$ is log bounded. By Lemma \ref{first step of relative MMP}, there exists a locally stable morphism $(\Cx,\Cb)\rightarrow \Cs$ such that $(X_g,B_g)$ is isomorphic to $(\Cx_{s_0},\Cb_{s_0})$ for a closed point $s_0\in \Cs$. By the proof of Theorem \ref{Main theorem 2 with boundary}, there exists $l'\in \mathbb{N}$ such that 
				$$l'(K_{X_g}+B_g)\sim 0.$$
				Also because $B$ is vertical over $Z$, then $\mathrm{coeff}(B)$ is in a DCC set $\frac{1}{l}\mathbb{N}\cap (0,1)$. By Theorem \ref{Main theorem 1}, we may assume
				\begin{itemize}
					\item $\mathrm{deg}(\pi)=m$.
				\end{itemize}
				
				Let $R$ be the ramified divisor of $\pi$, by Hurwitz formula, we have $K_{W}\sim \pi^*K_Z +R$. Because $f$ has reduced fibers in codimension 1, then the ramified divisor of $X_W\rightarrow X$ is just the pullback of the ramified divisor of $W\rightarrow Z$, which is $f_W^*R$, and we have
				$$K_{X_W}-f_W^*R\sim \pi_X^*K_X,$$
				where $\pi_X$ is the finite cover $X_W\rightarrow X$.
				
				Note $K_{X_g\times W}= p_1^*K_{X_g}+p_2^*K_W$, where $p_1,p_2$ are the projections $X_g\times W\rightarrow X_g, X_g\times W\rightarrow W$.
				Let $B_{X_g\times W}:=p_1^*B_g$, because $l'(K_{X_g}+B_g)\sim_{\mathbb{Q}} 0$, then we have
				$$l'(K_{X_g\times W}+B_{X_g\times W}-p_2^*R)\sim l'p_2^*(K_W-R).$$
				
				Because we assume Conjecture \ref{index conjecture} in dimension $d-1$ and $\mathrm{dim}(X_g)\leq d-1$, there exists $l''$ depending only on $d$ such that $l''K_Z\sim 0$. Then by Hurwitz formula, $l''(K_W-R)\sim 0$, we have
				\begin{equation}\label{equation a in the proof of Main corollary 1}
					l'l''(K_{X_g\times W}+B_{X_g\times W}-p_2^*R)\sim l'l''p_2^*(K_W-R)\sim 0.
				\end{equation}
				
				Let $B_W$ be the pullback of $B$ on $X_W$ and $B'_{X_g\times W}$ be its strict transform on $X_g\times W$. By Theorem \ref{Main theorem 1}, $(X_W,B_W)\rightarrow W$ is generically trivial. The pullback of $K_X+B\sim_{\mathbb{Q}} f^*K_Z\sim_{\mathbb{Q}} 0$ on $X_W$ gives
				\begin{equation}\label{equation b in the proof of Main corollary 1}
					K_{X_W}+B_W-f_W^*R\sim_{\mathbb{Q}} f_W^*\pi^*K_Z\sim_{\mathbb{Q}} 0.
				\end{equation}
				Because the strict transform of $f_W^*R$ on $X_g\times W$ is $p_2^*R$, then its strict transform of Equation \eqref{equation b in the proof of Main corollary 1} on $X_g\times W$ gives
				$$K_{X_g\times W}+B'_{X_g\times W}-p_2^*R\sim_{\mathbb{Q}} 0.$$
				In particular $B'_{X_g\times W}\sim_{\mathbb{Q}} B_{X_g\times W}$.
				
				Because the restriction of $B'_{X_g\times W}$ and $B_{X_g\times W}$ on a general fiber of $X_g\times W\rightarrow W$ are both $B_g$, $B_{X_g\times W}$ is trivial over $W$, and $(X_W,B_W)\rightarrow W$ is generically trivial, then $B'_{X_g\times W}$ is also trivial over $W$ and $B'_{X_g\times W}=B_{X_g\times W}$. The strict transform of Equation \ref{equation a in the proof of Main corollary 1} on $X_W$ is
				$$	l'l''(K_{X_W}+B_W-f_W^*R)\sim  0.$$
				
				Define $l:=l'l''m$. Because $\mathrm{deg}(\pi_X)=m$, then $(\pi_X)_*(K_{X_W}+B_W-f_W^*R)=m(K_X+B)$, where $(\pi_X)_*$ stands for the cycle-theoretic direct image. Thus $l'l''(K_{X_W}+B_W-f_W^*R)\sim  0$ implies $l(K_X+B)\sim 0$.

				Because $(X,B)\dashrightarrow (Y,B_Y)$ is a flop, then $l(K_X+B)\sim 0$ implies $l(K_Y+B_Y)\sim 0$.
			\end{proof}
		\end{thm}
		\begin{proof}[Proof of Corollary \ref{Main corollary 1}]
			This is a special case of Theorem \ref{Main corollary 1 with boundary} by taking $B=0$.
		\end{proof}

	\begin{acknowledgements}
		The author would like to thank his advisor Caucher Birkar for his encouragement and constant support. He would like to thank Stefano Filipazzi for insightful comments and providing Example \ref{Stefano's example}. The author also acknowledges Bingyi Chen, Santai Qu, Xiaowei Jiang, Jingjun Han, and Jihao Liu for their valuable comments. This work was supported by grants from Tsinghua University, Yau Mathematical Science Center. 	
	\end{acknowledgements}

\end{document}